 \documentclass[a4paper,11pt]{amsart}
\usepackage{amsmath,amsthm,amsfonts,amssymb,upgreek,multicol,mathrsfs,pifont,amscd,latexsym,cite,bm,geometry,color,fancyhdr,abstract,framed,appendix}
\usepackage{graphicx,algorithm,algorithmic,tabularx,url,booktabs,array,threeparttable,tikz}
\usepackage{enumitem}
\usepackage{comment}
\usetikzlibrary{calc}
\usetikzlibrary{angles,quotes}

\usepackage{caption}
\usepackage{subfig}
\usepackage{appendix}

\allowdisplaybreaks

\newcommand{\be}{\begin{eqnarray}}
\newcommand{\ee}{\end{eqnarray}}

\newcommand{\Kh}{\mathcal{T}_h}
\newcommand{\Khint}{\mathcal{T}_{h,\mathrm{int}}}
\newcommand{\Khb}{\mathcal{T}_{h,b}}

\newcommand{\IKh}{\mathcal{I}_h^{d}}
\newcommand{\IKhint}{\mathcal{I}_{h,\mathrm{int}}^{d}}
\newcommand{\IKhb}{\mathcal{I}_{h,b}^{d}}

\newcommand{\Fhdd}{\mathcal{F}_h^{d-1}}
\newcommand{\Fhddint}{\mathcal{F}_{h,\mathrm{int}}^{d-1}}
\newcommand{\Fhddb}{\mathcal{F}_{h,b}^{d-1}}
\newcommand{\Ehdd}{\mathcal{E}_h^{d-2}}
\newcommand{\Ehddint}{\mathcal{E}_{h,\mathrm{int}}^{d-2}}
\newcommand{\Ehddb}{\mathcal{E}_{h,b}^{d-2}}
\newcommand{\faceF}[2]{F_{#2}^{#1}}
\newcommand{\faceE}[2]{E_{#2}^{#1}}
\makeatletter
\newcommand{\Xint}[1]{\mathchoice
  {\XXint\displaystyle\textstyle{#1}}%
  {\XXint\textstyle\scriptstyle{#1}}%
  {\XXint\scriptstyle\scriptscriptstyle{#1}}%
  {\XXint\scriptscriptstyle\scriptscriptstyle{#1}}%
  \!\int}
\newcommand{\XXint}[3]{%
  \setbox0=\hbox{$#1{#2#3}{\int}$}%
  \vcenter{\hbox{$#2#3$}}\kern-.5\wd0}
\newcommand{\fint}{\Xint-}
\makeatother
\newcommand{\vsum}[1]{\Sigma_{#1}}

\newcommand{\ben}{\begin{equation*}}
\newcommand{\een}{\end{equation*}}

\newtheorem{theorem}{Theorem}[section]

\newtheorem{df}[theorem]{Definition}
\newtheorem{lemma}[theorem]{Lemma}

\newtheorem{remark}[theorem]{Remark}

\numberwithin{equation}{section}
\numberwithin{table}{section}
\numberwithin{figure}{section}

\begin{document}

\title[]{Piecewise $P_2$
 Finite Element Space on Rectangular Meshes in Arbitrary Dimensions for the Biharmonic Equation}

\thanks{The first author was supported by NSFC project No. 12301466 and by NSFC project No. 12422114.}


\author[R. Ma]{Rui Ma}
\address{School of Mathematics and Statistics, Beijing Institute of Technology, 
	Beijing 100081, P. R. China. }
\email{rui.ma@bit.edu.cn}

\author{Yang Xing}
\address{School of Mathematics and Statistics, Beijing Institute of Technology, 
	Beijing 100081, P. R. China. }
\email{xingyang@bit.edu.cn}
\author{Shuo Zhang}
\address{School of Mathematical Sciences, University of Chinese Academy of Sciences, Beijing 100049, China }
\email{szhang@lsec.cc.ac.cn}

\maketitle
\begin{abstract}
This paper develops a piecewise $P_2$ finite element space  on rectangular
meshes in arbitrary dimensions for the biharmonic equation.  The discrete space  imposes continuity of vertex sums on $(d-2)$-faces and of averaged normal derivatives across interior $(d-1)$-faces. A local algebraic characterization of $P_2$ is first established through the unisolvence of an enriched polynomial space and the resulting compatibility conditions. For each interior element, a $3^d$-element patch is then introduced, and the associated homogeneous patch space is shown to be one-dimensional. The corresponding patch functions, extended by zero, form a basis of the global homogeneous discrete space. A local quadratic interpolation operator is further constructed in arbitrary dimensions. Its approximation properties, together with a consistency estimate derived from the weak continuity conditions, lead to the optimal broken $H^2$ error estimate for the biharmonic problem.
\end{abstract}

 
\section{Introduction}
Nonconforming finite element methods provide an effective approach to fourth-order elliptic problems: by relaxing the global $H^2$ conformity requirement and replacing pointwise
continuity conditions on function values and derivatives with suitable weak
continuity conditions, these methods can substantially reduce the polynomial degree and the number of degrees of freedom while retaining optimal approximation properties. They have therefore been widely used for biharmonic-type problems.

On simplicial meshes, the theory of nonconforming methods is relatively well developed. The classical Morley element \cite{morley1967triangular} is the prototype for fourth-order problems. It employs quadratic polynomials with vertex values and edge normal-derivative moments as its degrees of freedom. This construction has been systematically generalized. In particular, Wang and Xu \cite{WangXu2006,wang2013minimal} developed a unified family of minimal nonconforming finite element spaces for $2m$-th order elliptic equations on simplicial meshes, with the Morley element as the lowest-degree case for fourth-order problems. More recently, Wu and Xu \cite{WuXu2019} and Li and Wu \cite{LiWu2026} have extended related constructions to arbitrary orders and dimensions.

The situation is different on quadrilateral and rectangular meshes. For second-order elliptic problems, low-order nonconforming methods based on piecewise linear polynomials have been studied extensively. Their global spaces are typically characterized by interelement moment conditions rather than pointwise continuity. Park and Sheen \cite{ParkSheen2003} proposed a $P_1$
 nonconforming element on general quadrilateral meshes based on edge moment continuity, while Hu and Shi \cite{HuShi2005} introduced a constrained nonconforming rotated $Q_1$
 element with three local degrees of freedom. Park and Sheen \cite{ParkSheen2005} considered extensions to rectangular meshes in arbitrary dimensions.

For fourth-order problems, far fewer constructions are available on rectangular meshes. The tensor-product geometry hinders the direct transfer of simplicial design principles. Classical examples include the Adini element \cite{AdiniNonFem}, the rectangular Morley element, and the family of $d$-rectangle Morley elements in \cite{WangShiXu2007}. On general convex quadrilateral meshes, Park and Sheen \cite{ParkSheen2013} proposed a Morley-type nonconforming element for the Kirchhoff plate problem. More recently, Zhang \cite{zhang2020minimalconsistent} constructed a minimal consistent finite element space consisting only of piecewise quadratic polynomials on two-dimensional quadrilateral grids. Quadratic-preserving interpolation techniques for such spaces were subsequently developed in two dimensions \cite{DaiZengZhangZhang2024}. A corresponding $P_2$ construction, which is of theoretically lowest degree, on rectangular meshes in arbitrary dimensions, however, has remained open.

This paper develops piecewise $P_2$ finite element space  on rectangular
meshes in arbitrary dimensions for the biharmonic equation. The resulting
space is nonconforming, with local shape functions consisting solely of
polynomials in $P_2$, without cubic or higher-order enrichment. The discrete
spaces impose continuity of vertex sums on $(d-2)$-faces and of averaged
normal derivatives across interior $(d-1)$-faces. In two dimensions, the
resulting homogeneous space coincides with the reduced quadrilateral Morley
space on rectangular meshes \cite{zhang2020minimalconsistent}.

The main contributions of this paper are as follows. First, a local algebraic characterization of $P_2$ on a $d$-dimensional rectangle is established by proving the unisolvence of an enriched polynomial space with respect to a suitable set of functionals and by deriving compatibility conditions that characterize characterize functions in $P_2$. Second, for each interior element, a $3^d$-element patch is introduced, and the associated homogeneous patch space is proved to be one-dimensional; the corresponding patch functions, extended by zero to the whole domain, are shown to form a basis of the global homogeneous space, yielding an explicit dimension formula. Third, a local quadratic-reproducing interpolation operator is constructed
on an extended rectangular mesh in arbitrary dimensions. The continuity conditions further imply continuity of  the face-averaged
gradient across interior $(d-1)$-faces, which yields the required consistency estimate. Together with the approximation properties of the interpolation operator, this gives, for $u\in H_0^2(\Omega)\cap H^3(\Omega)$,
\[ \|u-u_h\|_{2,h} \leq Ch\bigl(\|u\|_{3,\Omega}+h\|f\|_{0,\Omega}\bigr). \]
The resulting space also admits a natural interpretation as a nonconforming quadratic spline space on rectangular meshes: local quadratic polynomials are coupled through weak interelement continuity conditions, and the associated basis functions have compact patch supports.

The discrete space constructed in this paper is not of the classical
Ciarlet type. It has a natural spline structure: piecewise quadratic
polynomials are coupled through weak interelement continuity conditions,
and the corresponding basis functions have compact patch supports. The
construction thus provides a new type of nonconforming quadratic spline
on rectangular meshes, whose weak continuity structure makes a
lowest-degree $P_2$ discretization of the biharmonic problem possible in
arbitrary dimensions.

The remainder of the paper is organized as follows.
Section~\ref{sec:preli} introduces the notation, the discrete spaces, and the
main results, followed by a three-dimensional numerical experiment.
Section~\ref{sec:local-unisolvence-patch} establishes the local unisolvence
and compatibility properties and constructs the patch basis functions.
Section~\ref{sec:basis-interpolation-error} proves the global basis theorem,
develops the interpolation operator, and derives the corresponding
approximation and consistency estimates, which lead to the error estimate for
the biharmonic problem.
The auxiliary proofs are collected in the appendix.
\section{Preliminaries and main results}\label{sec:preli}
This section introduces the notation, the biharmonic model problem, and the rectangular mesh used throughout the paper. The piecewise $P_2$ finite element space  and the corresponding discrete problem are then defined, followed by the main theoretical results and a three-dimensional numerical experiment illustrating the convergence of the proposed method.

Let $\Omega$ be a rectangular domain in the $d$-dimensional Euclidean space $\mathbb{R}^d$, and denote its boundary by $\partial\Omega$. For a subdomain $D\subset\Omega$, $|D|$ denotes its measure. The average of an integrable function $g$ over $D$ is $\fint_D g:=|D|^{-1}\int_D g$. The notation $\#A$ denotes the number of elements in a finite set $A$. If $\Phi$ is a set of functionals, then $\Phi(v)=0$ means $\phi(v)=0$ for all $\phi\in\Phi$. The standard Sobolev spaces are denoted by $H^m(D)$ and $H^m_0(\Omega)$. 
The norm and seminorm on $H^m(D)$ are denoted by $\|\cdot\|_{m,D}$ and $|\cdot|_{m,D}$, respectively; when $m=0$, $\|\cdot\|_{0,D}$ is the $L^2(D)$ norm. Let $P_m(D)$ denote the space of polynomials on $D$ of total degree at most $m$.

Consider the model problem:
\begin{equation}
    \begin{cases}
\Delta^{2} u = f & \text{in }\Omega,\\[4pt]
u =  \dfrac{\partial u}{\partial \nu} = 0 & \text{on }\partial\Omega.
\end{cases}
\label{biharmonic}
\end{equation}
Given $f\in L^2(\Omega)$, the variational problem seeks
$u\in H_0^2(\Omega)$ such that
\begin{equation}
(\nabla^{2}u,\nabla^{2}v)=(f,v)\quad\text{for all }  v\in H_0^2(\Omega).
  \label{weak-biharmonic}
\end{equation}\subsection{Rectangular meshes and finite element spaces}\label{subsec:mesh}
Let $\mathcal{T}_h$ be a shape-regular rectangular mesh of $\Omega$. Along each coordinate direction $x_i$, $1\leq i\leq d$,
 the closure of $\Omega$ is subdivided by the
nodes $x_{i,1}<x_{i,2}<\cdots<x_{i,N_i+1}$. Set
$h_{i,r}:=x_{i,r+1}-x_{i,r}>0$ for $1\leq i\leq d$ and
$1\leq r\leq N_i$. Let $e_i$ be the $i$-th coordinate vector for $1\leq i\leq d$. Introduce the element index set 
$\IKh:=\{\alpha=(\alpha_1,\ldots,\alpha_d):1\leq\alpha_i\leq N_i,\
1\leq i\leq d\}$. For each $\alpha\in\IKh$, define
$K_{\alpha}:=\prod_{i=1}^{d}[x_{i,\alpha_i},x_{i,\alpha_i+1}]$, and
denote its barycenter by $x_{K_\alpha}$.
For $1\leq i\leq d$, the faces of $K_\alpha$ given by
$x_i=x_{i,\alpha_i}$ and $x_i=x_{i,\alpha_i+1}$ are called  the lower and upper $x_i$-faces of $K_\alpha$, respectively.
The closure of $\Omega$ is exactly the union of all elements
$K_\alpha$, and $\Kh=\{K_{\alpha}:\alpha\in\IKh\}$. The interior
and boundary element index sets are
$\IKhint:=\{\alpha\in\IKh:2\leq\alpha_i\leq N_i-1,\
1\leq i\leq d\}$ and
$\IKhb:=\{\alpha\in\IKh:\alpha_i\in\{1,N_i\}\text{ for at least one
}i\}$, respectively. The corresponding element sets are
$\Khint:=\{K_{\alpha}:\alpha\in\IKhint\}$ and
$\Khb:=\{K_{\alpha}:\alpha\in\IKhb\}$.Thus
$\Kh=\Khint\cup\Khb$, and each element in $\mathcal{T}_h$ is uniquely determined by its index $\alpha$. Let $h:=\max_{K\in\Kh}\operatorname{diam}K$.  For a mesh $\Kh$, the broken seminorm and norm are defined by $|v|_{m,h}^{2}:=\sum_{K\in\Kh}|v|_{m,K}^{2}$ and $\|v\|_{m,h}^{2}:=\sum_{K\in\Kh}\|v\|_{m,K}^{2}$, respectively. 

For every direction $1\leq i\leq d$, define the $(d-1)$-face index set $\mathcal{I}_{h,i}^{d-1}:=\{\beta=(\beta_1,\ldots,\beta_d):1\leq\beta_i\leq N_i+1,\ 1\leq\beta_m\leq N_m\text{ for }m\ne i\}$. For $\beta\in\mathcal{I}_{h,i}^{d-1}$, define  $ \faceF{i}{\beta}:=\{x\in\overline\Omega:x_i=x_{i,\beta_i},\ x_m\in[x_{m,\beta_m},x_{m,\beta_m+1}],\ m\ne i\}$, and denote its barycenter by $x_{F_\beta^i}$. Each index pair $(i,\beta)$ determines one $(d-1)$-face perpendicular to the $x_i$-direction. The set of all $(d-1)$-faces is denoted by $\Fhdd:=\{\faceF{i}{\beta}:1\leq i\leq d,\ \beta\in\mathcal{I}_{h,i}^{d-1}\}$. The interior and boundary $(d-1)$-face sets are denoted by $\Fhddint:=\{\faceF{i}{\beta}\in\Fhdd:2\leq\beta_i\leq N_i\}$ and $\Fhddb:=\{\faceF{i}{\beta}\in\Fhdd:\beta_i\in\{1,N_i+1\}\}$, respectively.

For a $(d-1)$-face $F\in\Fhdd$, let $\mathcal{V}(F)$ denote its set of vertices. For a continuous function $v$ on $F$, define the vertex-sum functional by $\vsum{F}(v):=\sum_{z\in\mathcal{V}(F)}v(z)$.
For every $(d-1)$-face $F\in\Fhdd$, fix a unit normal vector $\nu_F$; on a boundary face, $\nu_F$ is the outward unit normal. For a sufficiently smooth function $v$ on an element adjacent to $F$, define the normal-derivative average by $\fint_F \partial_{\nu_F}v\,ds:=|F|^{-1}\int_F \partial_{\nu_F}v\,ds$.

For $1\leq i\ne j\leq d$, define the $(d-2)$-face index set $\mathcal{I}_{h,i,j}^{d-2}:=\{\gamma=(\gamma_1,\ldots,\gamma_d):1\leq\gamma_i\leq N_i+1,\ 1\leq\gamma_j\leq N_j+1,\ 1\leq\gamma_m\leq N_m\text{ for }m\ne i,j\}$. For $\gamma\in\mathcal{I}_{h,i,j}^{d-2}$, define $\faceE{i,j}{\gamma}:=\{x\in\overline\Omega:x_i=x_{i,\gamma_i},\ x_j=x_{j,\gamma_j},\ x_m\in[x_{m,\gamma_m},x_{m,\gamma_m+1}],\ m\ne i,j\}$, and denote its barycenter by $x_{E_\gamma^{i,j}}$. Note that $\mathcal{I}_{h,i,j}^{d-2}=\mathcal{I}_{h,j,i}^{d-2}$ and
$E_\gamma^{i,j}=E_\gamma^{j,i}$. The set of all $(d-2)$-faces is denoted by  $\Ehdd:=\{\faceE{i,j}{\gamma}:1\leq i< j\leq d,\ \gamma\in\mathcal{I}_{h,i,j}^{d-2}\}$. The interior and boundary $(d-2)$-face sets are denoted by $\Ehddint:=\{\faceE{i,j}{\gamma}\in\Ehdd:2\leq\gamma_i\leq N_i,\ 2\leq\gamma_j\leq N_j\}$ and $\Ehddb:=\{\faceE{i,j}{\gamma}\in\Ehdd:\gamma_i\in\{1,N_i+1\}\text{ or }\gamma_j\in\{1,N_j+1\}\}$, respectively. 

For a $(d-2)$-face $E\in\Ehdd$, let $\mathcal{V}(E)$ denote its set of vertices. For a continuous function $v$ on $E$, define the vertex-sum functional by $\vsum{E}(v):=\sum_{z\in\mathcal{V}(E)}v(z)$. 

The piecewise quadratic finite element spaces on $\Kh$ are defined as follows.
\begin{itemize}
    \item The $d$-dimensional nonconforming $P_2$ rectangular finite element space
    \begin{equation}
    \begin{aligned}
        V_{h}
:= \Big\{ v_h& \in L^{2}(\Omega) \;:\;
v_h\big|_{K} \in P_{2}(K)\ \ \text{for all } K \in \Kh,\\
&\vsum{E}(v_h)\ \text{is continuous across every}\ E\in\Ehdd,\\
&\fint_{F} \frac{\partial v_h}{\partial \nu_F}\,ds
\ \text{is continuous across every}\ F\in\Fhddint\Big\}.
    \end{aligned}
    \label{ncp2-space}
 \end{equation}

 \item The homogeneous $d$-dimensional nonconforming $P_2$ rectangular finite element space
 \begin{equation}
     \begin{aligned}
         V_{h0}
:= \Big\{ v_h \in V_{h} \;:\;
\vsum{E}(v_h)=& 0 \ \text{on every}\ E\in\Ehddb,\\
\fint_{F} \frac{\partial v_h}{\partial \nu_F}\,ds=&0 \ \text{on every}\ F\in\Fhddb
\Big\}.
     \end{aligned}
 \label{ncp2-space-homogeneous}
 \end{equation}
\end{itemize}
When $d=2$, $V_{h0}$ coincides with the homogeneous reduced
quadrilateral Morley (RQM) finite element space in
\cite{zhang2020minimalconsistent} when restricted to rectangular meshes.

The finite element discretization for
\eqref{weak-biharmonic} is to find $u_h\in V_{h0}$ such that
\begin{equation}\label{eq:discrete-ncp2}
\sum_{K \in \Kh} \big( \nabla^{2} u_h, \nabla^{2} v \big)_{K}
= (f, v)\quad\text{for all }  v\in V_{h0}.
\end{equation} 
\subsection{Main results}\label{sec:main-top}
The following two main results establish an explicit basis characterization of $V_{h0}$ and the optimal error estimate for the discrete problem~\eqref{eq:discrete-ncp2}.  For each interior element $K_\alpha$, let
$\varphi_\alpha$ denote the corresponding patch function constructed in
Subsection~\ref{sec:basis-vh0}.

\begin{theorem}
\label{thm:dim-theory}
Let $\Omega$ be a rectangular domain in $\mathbb{R}^d$, let $\Kh$ be a
rectangular mesh of $\Omega$, and let $V_{h0}$ be the homogeneous
$d$-dimensional nonconforming $P_2$ rectangular finite element space
on $\Kh$. Then
\[
\{\varphi_\alpha:\alpha\in\IKhint\}
\]
is a basis of $V_{h0}$. Consequently,
\[
\dim V_{h0}=\#\Khint.
\]
\end{theorem}

The proof is given in Subsection~\ref{sec:proof-dim-theory}.

\begin{theorem}
\label{thm:error}
Let $\Kh$ be a shape-regular rectangular mesh of $\Omega$, and let $u$ and $u_h$
be the solutions of \eqref{weak-biharmonic} and
\eqref{eq:discrete-ncp2}, respectively. If $u\in H^3(\Omega)$, then
\begin{equation}
\label{eq:4.6}
\|u-u_h\|_{2,h}
\leq Ch\bigl(\|u\|_{3,\Omega}+h\|f\|_{0,\Omega}\bigr).
\end{equation}
\end{theorem}
\begin{proof} The continuity and homogeneous boundary conditions in
\eqref{ncp2-space-homogeneous} imply that $|\cdot|_{2,h}$ is a norm on
$V_{h0}$. 
Apply the second Strang lemma  \cite{brennerscott} to the discrete problem
\eqref{eq:discrete-ncp2}. The approximation term is controlled by
Theorem~\ref{thm:pi-h0-error} and the consistency term by
Theorem~\ref{thm:consistency-error}.
\end{proof}

\subsection{Numerical experiment in three dimensions}
\label{sec:numerical-example}

This subsection presents a three-dimensional numerical experiment to verify the
convergence of the proposed nonconforming $P_2$ method. Consider the
biharmonic problem \eqref{biharmonic} on
\[
\Omega=(0,2)\times(-1,1)\times(0,1).
\]
The exact solution is chosen as
\[
u(x,y,z)
=
x^2(2-x)^2
(y+1)^2(1-y)^2
z^2(1-z)^2,
\]
and the right-hand side is defined by $f=\Delta^2u$. Clearly,
$u\in H_0^2(\Omega)$ and is sufficiently smooth.

The initial rectangular mesh, shown in
Figure~\ref{fig:initial-3d-mesh}, consists of
$N_1\times N_2\times N_3=4\times6\times5$ elements.
Each refinement bisects every element in all three coordinate directions.
\begin{figure}[htbp]
    \centering
\includegraphics[width=0.45\textwidth]{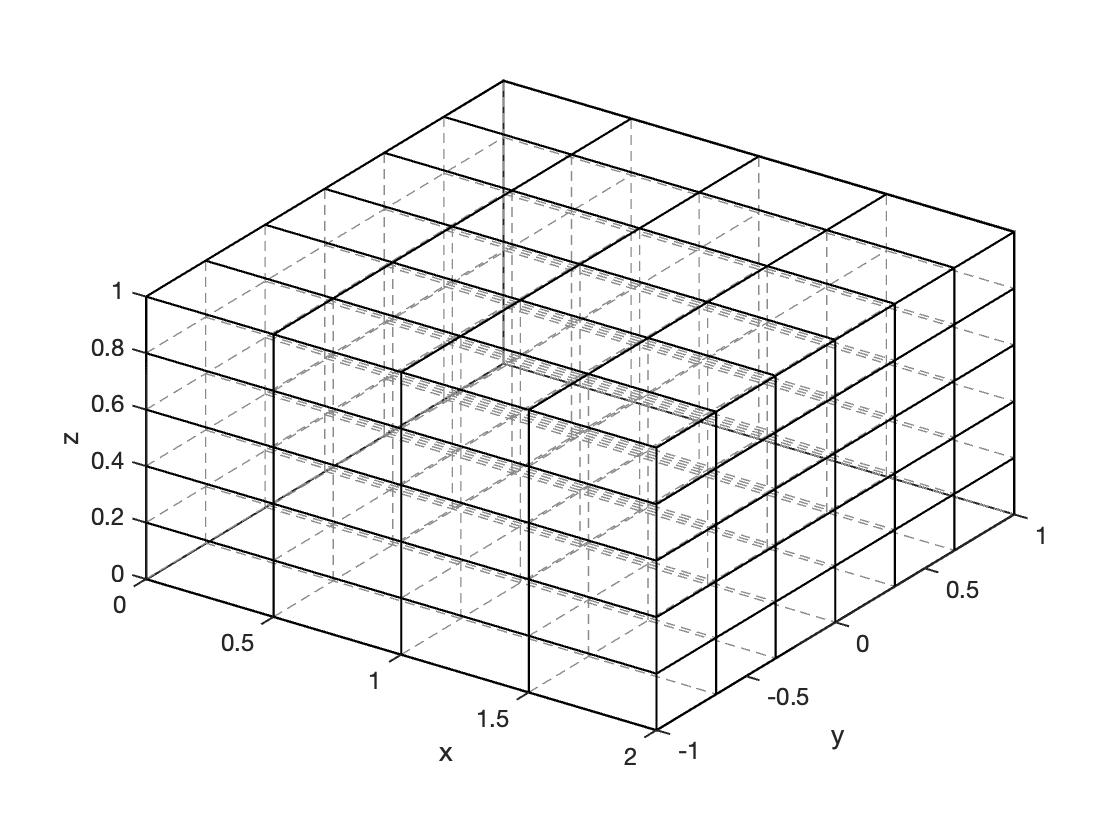}
    \caption{Initial rectangular mesh of
    $\Omega=(0,2)\times(-1,1)\times(0,1)$.}
    \label{fig:initial-3d-mesh}
\end{figure}
By Theorem~\ref{thm:dim-theory}, the number of degrees of freedom is
\[
\dim V_{h0}=(N_1-2)(N_2-2)(N_3-2).
\]

The errors are measured in the $L^2$ norm, the broken $H^1$ norm, and
the broken $H^2$ seminorm,
\[
\|v\|_{1,h}^2
:=
\sum_{K\in\Kh}\|v\|_{1,K}^2,
\qquad
|v|_{2,h}^2
:=
\sum_{K\in\Kh}|v|_{2,K}^2.
\]
The numerical results are reported in Table~\ref{tab:3d-convergence}.

\begin{table}[htbp]
\centering
\caption{Errors and convergence rates for the three-dimensional biharmonic
problem.}
\label{tab:3d-convergence}
\begin{tabular}{cccccccccc}
\hline
$N_1$ & $N_2$ & $N_3$ & $\dim V_{h0}$
& $\|u-u_h\|_{0,\Omega}$ & rate
& $\|u-u_h\|_{1,h}$ & rate
& $|u-u_h|_{2,h}$ & rate \\
\hline
4  & 6  & 5  & 24
& $4.8050\times10^{-3}$ & --
& $2.9750\times10^{-2}$ & --
& $4.8862\times10^{-1}$ & -- \\
8  & 12 & 10 & 480
& $1.3139\times10^{-3}$ & 1.871
& $7.8250\times10^{-3}$ & 1.927
& $2.4256\times10^{-1}$ & 1.010 \\
16 & 24 & 20 & 5544
& $3.3185\times10^{-4}$ & 1.985
& $1.9719\times10^{-3}$ & 1.989
& $1.2074\times10^{-1}$ & 1.006 \\
32 & 48 & 40 & 52440
& $8.3148\times10^{-5}$ & 1.997
& $4.9389\times10^{-4}$ & 1.997
& $6.0291\times10^{-2}$ & 1.002 \\
\hline
\end{tabular}
\end{table}

The numerical results show first-order convergence in the broken $H^2$
seminorm, in agreement with the estimate in
Theorem~\ref{thm:error}. For this smooth solution, the
$L^2$ and broken $H^1$ errors exhibit approximately second-order
convergence. These results confirm the predicted convergence of the
three-dimensional nonconforming $P_2$ method and also provide a numerical
verification of the patch-basis implementation.

\section{Local unisolvence and patch basis}\label{sec:local-unisolvence-patch}
This section develops the local results used to construct a basis of $V_{h0}$. The first two subsections establish unisolvence and compatibility results for quadratic polynomials on a single $d$-dimensional rectangle. The third subsection introduces the $3^d$-element patches and constructs the corresponding patch functions. The final subsection proves the global basis theorem.

\subsection{Unisolvence for the enriched space $P^+$ on a $d$-dimensional rectangle}

This subsection establishes the local algebraic characterization of $P_2(K)$ on a $d$-dimensional rectangle $K$. The enriched space $P_2(K)+\operatorname{span}\{x_i^3\}_{i=1}^d$ is considered first and shown to be unisolvent with respect to the functionals in $\mathcal N^+$ defined below. Compatibility conditions are then derived to characterize functions in $P_2(K)$.

Without loss of generality, consider $K=K_{\alpha}=\prod_{i=1}^{d}[x_{i,1},x_{i,2}]$ with $\alpha=(1,1,\ldots,1)$. The arguments below apply to every element $K_{\alpha}\in\Kh$.
 Since only one interval is used in each direction, the following simplified notation is adopted throughout this section. 
For $1\leq i\leq d$ and $k\in\{1,2\}$, set the $(d-1)$-face $F^{i,k}:=\{x\in K:x_i=x_{i,k}\}$. For $1\leq i\ne j\leq d$ and $k,l\in\{1,2\}$, set the $(d-2)$-face $E^{i,k;j,l}:=F^{i,k}\cap F^{j,l}$. Note that $E^{i,k;j,l}=E^{j,l;i,k}$, so the ordering is immaterial. Recall from Subsection~\ref{subsec:mesh} that $\Sigma_{F^{i,k}}$ and $\Sigma_{E^{i,k;j,l}}$ denote the corresponding vertex-sum functionals. A $(d-2)$-face vertex sum is said to be \emph{on} a $(d-1)$-face if the corresponding $(d-2)$-face is contained in it. Moreover, for $j\ne i$ and $k\in\{1,2\}$,
\begin{equation}
\label{eq:face-sum-edge-decomposition}
\Sigma_{F^{i,k}}(v)
=\Sigma_{E^{i,k;j,1}}(v)+\Sigma_{E^{i,k;j,2}}(v).
\end{equation}
This identity will be used repeatedly below.

For $1\leq i\leq d-1$, define
$\mathcal N_i:=\left\{\Sigma_{E^{i,1;i+1,2}}\right\}
\cup\bigcup_{j=i+1}^{d}\left\{\Sigma_{E^{i,1;j,1}}\right\}$,
and define
$\mathcal N_d:=\left\{\Sigma_{E^{1,2;d,1}}\right\}$. 
The functional set $\mathcal N$ is
$\mathcal N:=\left\{\Sigma_{E^{1,2;2,2}}\right\}
\cup\bigcup_{i=1}^{d}\mathcal N_i$. Note that \eqref{eq:face-sum-edge-decomposition} implies
\begin{align}\label{eq:span}
\operatorname{span}\bigcup_{i=1}^{d}\mathcal N_i
=
\operatorname{span}
\bigcup_{i=1}^{d}
\left(
\left\{\Sigma_{F^{i,1}}\right\}
\cup
\bigcup_{j=i+1}^{d}
\left\{\Sigma_{E^{i,1;j,1}}\right\}
\right).
\end{align}
For $d\geq2$, define the enriched space $P^{+}:=P_2(K)+\operatorname{span}\{x_i^3\}_{i=1}^{d}$, with $\dim P^{+}=(d+1)(d+2)/2+d$. Define the functional set $\mathcal N^{+}:=\mathcal N\cup\{\fint_{F^{i,k}}\partial v/\partial x_i\,ds\}_{1\leq i\leq d,\ k\in\{1,2\}}$. Here and below, $\fint_{F^{i,k}}\partial v/\partial x_i\,ds$ is called the coordinate-derivative average on $F^{i,k}$; up to the orientation sign, it is the normal-derivative average introduced in Subsection~\ref{subsec:mesh}.
\begin{remark}\label{rek:dimNplus}
The number of functionals in $\mathcal N^{+}$ equals
\[
    \#\mathcal N^{+}
    =\#\mathcal N+2d
    =1+\sum_{i=1}^{d}(d+1-i)+2d
    =\frac{d^2+d+2}{2}+2d
    =\dim P^{+}.
\]
\end{remark}
The proof of unisolvence begins with a propagation property of the vertex-sum functionals in $\mathcal N$: their vanishing forces the vertex sums on all $(d-2)$-faces to vanish.

\begin{lemma}
\label{lem:rm-face-sums}
For any function $v$ for which the vertex values are defined, if $\mathcal N(v)=0$, then
\[
    \Sigma_{E^{i,k;j,l}}(v)=0
    \qquad
    \text{for }1\leq i\ne j\leq d,\quad k,l\in\{1,2\}.
\]
\end{lemma}

\begin{proof}
\setlength{\abovedisplayskip}{5pt plus 2pt minus 2pt}
\setlength{\belowdisplayskip}{5pt plus 2pt minus 2pt}
\setlength{\abovedisplayshortskip}{3pt plus 2pt}
\setlength{\belowdisplayshortskip}{3pt plus 2pt minus 1pt}

\begin{figure}[H]
\centering
\begin{tikzpicture}[
  x=1cm,y=1cm,
  direct/.style={black,very thick},
  inherited/.style={red!75!black,very thick},
  deduced/.style={blue!75!black,very thick},
  facedirect/.style={fill=black!9},
  facededuced/.style={fill=blue!10},
  flow/.style={->,>=stealth,semithick},
  cubelabel/.style={font=\tiny,fill=white,inner sep=0.5pt},
  edgelabel/.style={
    font=\tiny,
    scale=0.68,
    transform shape,
    inner sep=0.2pt
  },
  facelabel/.style={font=\tiny},
  facetitle/.style={font=\scriptsize\bfseries}
]

\coordinate (A)  at (0.00,0.15);
\coordinate (B)  at (1.65,0.15);
\coordinate (C)  at (1.65,1.80);
\coordinate (D)  at (0.00,1.80);
\coordinate (Ap) at (0.70,0.70);
\coordinate (Bp) at (2.35,0.70);
\coordinate (Cp) at (2.35,2.35);
\coordinate (Dp) at (0.70,2.35);

\draw[semithick] (A)--(B)--(C)--(D)--cycle;
\draw[semithick] (B)--(Bp)--(Cp)--(C);
\draw[semithick] (D)--(Dp)--(Cp);
\draw[dashed,gray] (A)--(Ap)--(Bp);
\draw[dashed,gray] (Ap)--(Dp);

\node[cubelabel] at (0.82,1.10) {$F^{1,1}$};
\node[cubelabel] at (2.02,1.34) {$F^{2,2}$};
\node[cubelabel] at (1.16,2.08) {$F^{3,2}$};
\node[cubelabel] at (0.25,1.38) {$F^{2,1}$};
\node[cubelabel] at (1.47,1.55) {$F^{1,2}$};
\node[cubelabel] at (1.05,0.47) {$F^{3,1}$};


\begin{scope}[shift={(4.35,1.65)}]
  \node[facetitle] at (0,1.13) {(1)};

  \fill[facedirect] (-0.80,-0.80) rectangle (0.80,0.80);

  \draw[direct]  (-0.82,-0.82)--( 0.82,-0.82);
  \draw[deduced] (-0.82, 0.82)--( 0.82, 0.82);
  \draw[direct]  (-0.82,-0.82)--(-0.82, 0.82);
  \draw[deduced] ( 0.82,-0.82)--( 0.82, 0.82);

  \node[facelabel,black] at (0,0)
    {$\Sigma_{F^{1,1}}(v)=0$};

  \node[edgelabel,black,below] at (0,-0.82)
    {$\Sigma_{E^{1,1;2,1}}(v)=0$};

  \node[edgelabel,blue!75!black,above] at (0,0.82)
    {$\Sigma_{E^{1,1;2,2}}(v)=0$};

  \node[edgelabel,black,rotate=90,above] at (-0.82,0)
    {$\Sigma_{E^{1,1;3,1}}(v)=0$};

  \node[edgelabel,blue!75!black,rotate=90,below] at (0.82,0)
    {$\Sigma_{E^{1,1;3,2}}(v)=0$};
\end{scope}

\begin{scope}[shift={(6.85,1.65)}]
  \node[facetitle] at (0,1.13) {(2)};

  \fill[facedirect] (-0.80,-0.80) rectangle (0.80,0.80);

  \draw[direct]  (-0.82,-0.82)--( 0.82,-0.82);
  \draw[deduced] (-0.82, 0.82)--( 0.82, 0.82);
  \draw[direct]  (-0.82,-0.82)--(-0.82, 0.82);
  \draw[deduced] ( 0.82,-0.82)--( 0.82, 0.82);

  \node[facelabel,black] at (0,0)
    {$\Sigma_{F^{2,1}}(v)=0$};

  \node[edgelabel,black,below] at (0,-0.82)
    {$\Sigma_{E^{1,1;2,1}}(v)=0$};

  \node[edgelabel,blue!75!black,above] at (0,0.82)
    {$\Sigma_{E^{1,2;2,1}}(v)=0$};

  \node[edgelabel,black,rotate=90,above] at (-0.82,0)
    {$\Sigma_{E^{2,1;3,1}}(v)=0$};

  \node[edgelabel,blue!75!black,rotate=90,below] at (0.82,0)
    {$\Sigma_{E^{2,1;3,2}}(v)=0$};
\end{scope}

\begin{scope}[shift={(9.35,1.65)}]
  \node[facetitle] at (0,1.13) {(3)};

  \fill[facedirect] (-0.80,-0.80) rectangle (0.80,0.80);

  \draw[direct]  (-0.82,-0.82)--( 0.82,-0.82);
  \draw[deduced] (-0.82, 0.82)--( 0.82, 0.82);
  \draw[direct]  (-0.82,-0.82)--(-0.82, 0.82);
  \draw[deduced] ( 0.82,-0.82)--( 0.82, 0.82);

  \node[facelabel,black] at (0,0)
    {$\Sigma_{F^{3,1}}(v)=0$};

  \node[edgelabel,black,below] at (0,-0.82)
    {$\Sigma_{E^{1,1;3,1}}(v)=0$};

  \node[edgelabel,blue!75!black,above] at (0,0.82)
    {$\Sigma_{E^{1,2;3,1}}(v)=0$};

  \node[edgelabel,black,rotate=90,above] at (-0.82,0)
    {$\Sigma_{E^{2,1;3,1}}(v)=0$};

  \node[edgelabel,blue!75!black,rotate=90,below] at (0.82,0)
    {$\Sigma_{E^{2,2;3,1}}(v)=0$};
\end{scope}


\begin{scope}[shift={(4.35,-0.55)}]
  \node[facetitle] at (0,1.13) {(4)};

  \fill[facededuced] (-0.80,-0.80) rectangle (0.80,0.80);

  \draw[inherited] (-0.82,-0.82)--( 0.82,-0.82);
  \draw[direct]    (-0.82, 0.82)--( 0.82, 0.82);
  \draw[inherited] (-0.82,-0.82)--(-0.82, 0.82);
  \draw[deduced]   ( 0.82,-0.82)--( 0.82, 0.82);

  \node[facelabel,blue!75!black] at (0,0)
    {$\Sigma_{F^{1,2}}(v)=0$};

  \node[edgelabel,red!75!black,below] at (0,-0.82)
    {$\Sigma_{E^{1,2;2,1}}(v)=0$};

  \node[edgelabel,black,above] at (0,0.82)
    {$\Sigma_{E^{1,2;2,2}}(v)=0$};

  \node[edgelabel,red!75!black,rotate=90,above] at (-0.82,0)
    {$\Sigma_{E^{1,2;3,1}}(v)=0$};

  \node[edgelabel,blue!75!black,rotate=90,below] at (0.82,0)
    {$\Sigma_{E^{1,2;3,2}}(v)=0$};
\end{scope}

\begin{scope}[shift={(6.85,-0.55)}]
  \node[facetitle] at (0,1.13) {(5)};

  \fill[facededuced] (-0.80,-0.80) rectangle (0.80,0.80);

  \draw[inherited] (-0.82,-0.82)--( 0.82,-0.82);
  \draw[direct]    (-0.82, 0.82)--( 0.82, 0.82);
  \draw[inherited] (-0.82,-0.82)--(-0.82, 0.82);
  \draw[deduced]   ( 0.82,-0.82)--( 0.82, 0.82);

  \node[facelabel,blue!75!black] at (0,0)
    {$\Sigma_{F^{2,2}}(v)=0$};

  \node[edgelabel,red!75!black,below] at (0,-0.82)
    {$\Sigma_{E^{1,1;2,2}}(v)=0$};

  \node[edgelabel,black,above] at (0,0.82)
    {$\Sigma_{E^{1,2;2,2}}(v)=0$};

  \node[edgelabel,red!75!black,rotate=90,above] at (-0.82,0)
    {$\Sigma_{E^{2,2;3,1}}(v)=0$};

  \node[edgelabel,blue!75!black,rotate=90,below] at (0.82,0)
    {$\Sigma_{E^{2,2;3,2}}(v)=0$};
\end{scope}

\begin{scope}[shift={(9.35,-0.55)}]
  \node[facetitle] at (0,1.13) {(6)};

  \fill[facededuced] (-0.80,-0.80) rectangle (0.80,0.80);

  \draw[inherited] (-0.82,-0.82)--( 0.82,-0.82);
  \draw[inherited] (-0.82, 0.82)--( 0.82, 0.82);
  \draw[inherited] (-0.82,-0.82)--(-0.82, 0.82);
  \draw[inherited] ( 0.82,-0.82)--( 0.82, 0.82);

  \node[facelabel,blue!75!black] at (0,0)
    {$\Sigma_{F^{3,2}}(v)=0$};

  \node[edgelabel,red!75!black,below] at (0,-0.82)
    {$\Sigma_{E^{1,1;3,2}}(v)=0$};

  \node[edgelabel,red!75!black,above] at (0,0.82)
    {$\Sigma_{E^{1,2;3,2}}(v)=0$};

  \node[edgelabel,red!75!black,rotate=90,above] at (-0.82,0)
    {$\Sigma_{E^{2,1;3,2}}(v)=0$};

  \node[edgelabel,red!75!black,rotate=90,below] at (0.82,0)
    {$\Sigma_{E^{2,2;3,2}}(v)=0$};
\end{scope}



\draw[direct] (-0.40,-2.38)--(0.15,-2.38);
\node[font=\tiny,anchor=west] at (0.22,-2.38)
  {first step by $\mathcal N(v)=0$};

\draw[inherited] (3.4,-2.38)--(3.95,-2.38);
\node[font=\tiny,anchor=west] at (4.0,-2.38)
  {obtained in a previous step};

\draw[deduced] (7.90,-2.38)--(8.45,-2.38);
\node[font=\tiny,anchor=west] at (8.52,-2.38)
  {deduced from \eqref{eq:face-sum-edge-decomposition}};

\end{tikzpicture}

\caption{Illustration of the propagation argument in Lemma~\ref{lem:rm-face-sums} for $d=3$.}
\label{fig:three-dimensional-face-sum-propagation}
\end{figure}
Suppose that $\mathcal N(v)=0$. For the faces with $k=1$, \eqref{eq:span}  shows that $\Sigma_{F^{i,1}}(v)=0
$ for $1\leq i\leq d$ and $\Sigma_{E^{i,1;j,1}}(v)=0$ for $ 1\le i<j\le d$.  This and \eqref{eq:face-sum-edge-decomposition} show  \[ \Sigma_{E^{i,1;j,l}}(v) = \Sigma_{E^{j,l;i,1}}(v)=0, \qquad 1\le i\ne j\le d,\quad l\in\{1,2\}. \] For the faces with $k=2$, the previous result $\Sigma_{E^{1,2;2,1}}(v)=0$ plus $\Sigma_{E^{1,2;2,2}}(v)=0$ and \eqref{eq:face-sum-edge-decomposition} lead to $\Sigma_{F^{1,2}}(v)=0$, and hence \[ \Sigma_{E^{1,2;j,2}}(v)=0, \qquad 2\le j\le d. \] For $2\le i\le d$, \[ \Sigma_{E^{1,1;i,2}}(v) = \Sigma_{E^{1,2;i,2}}(v)=0. \]This and  \eqref{eq:face-sum-edge-decomposition}  give  $\Sigma_{F^{i,2}}(v)=0$. Together with $\Sigma_{E^{i,2;j,1}}(v)=0$ for $ j\ne i$, this yields \[ \Sigma_{E^{i,2;j,2}}(v)=0, \qquad 1\le i\ne j\le d. \] Therefore, \[ \Sigma_{E^{i,k;j,l}}(v)=0, \qquad 1\le i\ne j\le d,\quad k,l\in\{1,2\}. \] This concludes the proof. \end{proof}
\begin{remark}
Figure~\ref{fig:three-dimensional-face-sum-propagation} illustrates the argument in the proof of Lemma \ref{lem:rm-face-sums} for $d=3$. 
\end{remark}
 
The following two lemmas present elementary properties of affine functions on rectangles.
\begin{lemma}
\label{lem:p1-face-sum-vertex}
Let $d\geq1$, let $K$ be a $d$-dimensional rectangle, and let $x_K$ denote the centroid of $K$. Then, for any $v\in P_1(K)$,
\[
    \sum_{z\in\mathcal{V}(K)}v(z)=2^d v(x_K),
\]
where $\mathcal{V}(K)$ denotes the set of vertices of $K$.
\end{lemma}

\begin{proof}
The conclusion follows directly from
$x_K=2^{-d}\sum_{z\in\mathcal{V}(K)}z$.
\end{proof}

\begin{lemma}
\label{lem:p1-zero-face-sums}
Let $d\geq1$, and let $K$ be a $d$-dimensional rectangle. For any $v\in P_1(K)$, if
\[
    \Sigma_{F^{i,k}}(v)=0
    \qquad
    \text{for all }1\leq i\leq d,\ k\in\{1,2\},
\]
then $v= 0$.
\end{lemma}

\begin{proof}
Recall that $x_{F^{i,k}}$ denotes the midpoint of $F^{i,k}$. Applying 
Lemma~\ref{lem:p1-face-sum-vertex} in dimension $d-1$ shows 
\[
v(x_{F^{i,k}})=0,
\qquad 1\le i\le d,\quad k\in\{1,2\}.
\]
Suppose
$
v=c_0+\sum_{i=1}^d c_i x_i
$. Note that
$
x_{F^{i,2}}-x_{F^{i,1}}=(x_{i,2}-x_{i,1}) e_i 
$
lead to
\[
0
=
v(x_{F^{i,2}})-v(x_{F^{i,1}})
=
c_i (x_{i,2}-x_{i,1}) ,
\qquad 1\le i\le d.
\]
This shows $c_i=0$ for all $1\le i\le d$. Since
$v(x_{F^{1,1}})=0$, it follows that $c_0=0$. This concludes the proof.
\end{proof}
The previous lemmas provide the ingredients for establishing the unisolvence of the enriched space $P^+$ with respect to $\mathcal N^+$.

\begin{theorem}
\label{thm:rm-unisolvence}
For $d\geq2$, the enriched space $P^{+}$ is uniquely determined by $\mathcal N^{+}$.
\end{theorem}

\begin{proof}
By the dimension count in Remark \ref{rek:dimNplus}, it suffices to show
that $\mathcal N^+(v)=0$ implies $v=0$. Let
\[
v=\sum_{i=1}^d a_i x_i^3
 +\sum_{1\le i\le j\le d} b_{ij}x_ix_j
 +\sum_{i=1}^d c_i x_i+c_0
\]
with $b_{ij}=b_{ji}$.

Fix $1\le i\le d$ and define, on $F^{i,1}$, $ 
\delta_i v(x)
:=
v(x)-v\bigl(x+(x_{i,2}-x_{i,1})e_i\bigr) 
$. Then $\delta_i v\in P_1(F^{i,1})$. 
Lemma~\ref{lem:rm-face-sums} shows
\[
\Sigma_{E^{i,1;j,l}}(v)
=
\Sigma_{E^{i,2;j,l}}(v)=0,
\qquad j\ne i,\quad l\in\{1,2\}.
\]
This implies that $\delta_i v$ has zero vertex sum on every $(d-2)$-face of
$F^{i,1}$. The application of Lemma~\ref{lem:p1-zero-face-sums}  in dimension
$d-1$ gives
\begin{align}\label{eq:deltav}
\delta_i v=0
\qquad\text{on }F^{i,1}.
\end{align}
This shows
$
b_{ij}=0
$ for any $j\neq i$.
Since $i$ is arbitrary, all mixed terms vanish. Moreover, let
$
p_i(t):=a_i t^3+b_{ii}t^2+c_i t
$ and in fact $v=c_0+\sum_{i=1}^dp_i(x_i)$. Note that \eqref{eq:deltav} shows
\begin{equation}
\label{eq:pi-endpoint-values}
p_i(x_{i,1})=p_i(x_{i,2}),
\qquad 1\le i\le d.
\end{equation}
The two coordinate-derivative average conditions on $F^{i,1}$ and $F^{i,2}$
reduce to
\[
p_i'(x_{i,1})=p_i'(x_{i,2})=0.
\]
Together with \eqref{eq:pi-endpoint-values}, an elementary calculation yields
\[
a_i=b_{ii}=c_i=0,
\qquad 1\le i\le d.
\]
The previous arguments show $v=c_0$. 
This and 
$
\Sigma_{E^{1,1;2,1}}(v)=0
$ 
lead to $c_0=0$. This concludes the proof.
\end{proof}

\subsection{Unisolvence for $P_2(K)$ on a $d$-dimensional rectangle}\label{sec:Ptwouni}
 
It remains to characterize  functions in the subspace $P_2(K)\subset P^+$. The following identity gives a compatibility relation between the $(d-2)$-face vertex sums and the coordinate-derivative averages on the $(d-1)$-faces.

\begin{lemma}
\label{lem:p2-b-a-compatibility}
Let $K$ be a $d$-dimensional rectangle and let $v\in P_2(K)$. Then,
for every $1\leq i\leq d$ and $q\neq i$,
\begin{align}\label{eq:assum}
\displaystyle
\sum_{l=1}^2
\frac{
\Sigma_{E^{i,1;q,l}}(v)-\Sigma_{E^{i,2;q,l}}(v)
}{x_{i,1}-x_{i,2}}
=2^{d-2}\left(
\fint_{F^{i,1}}\frac{\partial v}{\partial x_i}\,ds
+\fint_{F^{i,2}}\frac{\partial v}{\partial x_i}\,ds
\right),
\end{align}
i.e.
\begin{align*}
\frac{
\Sigma_{F^{i,1}}(v)-\Sigma_{F^{i,2}}(v)
}{x_{i,1}-x_{i,2}}
=2^{d-2}\left(
\fint_{F^{i,1}}\frac{\partial v}{\partial x_i}\,ds
+\fint_{F^{i,2}}\frac{\partial v}{\partial x_i}\,ds
\right).
\end{align*}
\end{lemma}

\begin{proof}
Let
\[
w:=\frac{\partial v}{\partial x_i}\in P_1(K).
\]
The exact centered-difference formula in the $x_i$-direction, together with Lemma~\ref{lem:p1-face-sum-vertex} applied on a $(d-2)$ dimension face, leads to 
\[
\frac{
\Sigma_{E^{i,1;q,l}}(v)-\Sigma_{E^{i,2;q,l}}(v)
}{
x_{i,1}-x_{i,2}
}
=
2^{d-2}w(x_{F^{q,l}})\quad\text{ for } l=1,2.
\]
Hence
\[
\sum_{l=1}^2
\frac{
\Sigma_{E^{i,1;q,l}}(v)-\Sigma_{E^{i,2;q,l}}(v)
}{
x_{i,1}-x_{i,2}
}
=
2^{d-2}\bigl(w(x_{F^{q,1}})+w(x_{F^{q,2}})\bigr).
\]
Since $w$ is affine,
\[
w(x_{F^{q,1}})+w(x_{F^{q,2}})
=
w(x_{F^{i,1}})+w(x_{F^{i,2}})
=
\fint_{F^{i,1}}w\,ds+\fint_{F^{i,2}}w\,ds.
\]
This concludes the proof.
\end{proof}
Using the previous lemma, the following theorem characterizes the compatible vertex-sum and coordinate-derivative average data for $P_2(K)$.
\begin{theorem}
\label{lem:p2-data-compatibility}
Let $d\geq2$ and let $K$ be a $d$-dimensional rectangular element.
Given two sets of real numbers
$
\{a^{i,k;j,l}\}_{1\leq i\ne j\leq d,\ k,l\in\{1,2\}}$ and $
\{b^{i,k}\}_{1\leq i\leq d,\ k\in\{1,2\}},
$
satisfying the symmetry
\[
a^{i,k;j,l}=a^{j,l;i,k},
\qquad i\ne j,
\]
there exists a unique $v\in P_2(K)$ such that
\begin{equation}
\label{eq:p2-data}
\Sigma_{E^{i,k;j,l}}(v)=a^{i,k;j,l},
\qquad
\fint_{F^{i,k}}\frac{\partial v}{\partial x_i}\,ds=b^{i,k},
\end{equation}
for $1\leq i\ne j\leq d$ and $k,l\in\{1,2\}$, if and only if the
following compatibility conditions hold:
\begin{enumerate}[label=(\arabic*)]
    \item For $1\leq i\leq d$, $k\in\{1,2\}$, and $j,q\ne i$,
    \begin{equation}
    \label{eq:a-compatibility}
    a^{i,k;j,1}+a^{i,k;j,2}
    =
    a^{i,k;q,1}+a^{i,k;q,2}.
    \end{equation}

    \item For $1\leq i\leq d$ and any $q\ne i$,
    \begin{equation}
    \label{eq:b-a-compatibility}
    \sum_{l=1}^2
    \frac{
    a^{i,1;q,l}-a^{i,2;q,l}
    }{
    x_{i,1}-x_{i,2}
    }
    =
    2^{d-2}\bigl(b^{i,1}+b^{i,2}\bigr).
    \end{equation}
\end{enumerate}
\end{theorem}

\begin{proof}
The necessity follows directly from
\eqref{eq:face-sum-edge-decomposition} and
Lemma~\ref{lem:p2-b-a-compatibility}.

For the converse, suppose that
\eqref{eq:a-compatibility}--\eqref{eq:b-a-compatibility} hold.
Prescribe the values of the functionals in $\mathcal N^+$ by
\[
\Sigma_{E^{i,1;i+1,2}}
\mapsto a^{i,1;i+1,2},
\qquad 1\leq i\leq d-1,
\]
\[
\Sigma_{E^{i,1;j,1}}
\mapsto a^{i,1;j,1},
\qquad 1\leq i<j\leq d,
\]
together with
\[
\Sigma_{E^{1,2;d,1}}
\mapsto a^{1,2;d,1},
\qquad
\Sigma_{E^{1,2;2,2}}
\mapsto a^{1,2;2,2},
\]
and
\[
\fint_{F^{i,k}}\frac{\partial v}{\partial x_i}\,ds
\mapsto b^{i,k},
\qquad
1\leq i\leq d,\quad k\in\{1,2\}.
\]
Theorem~\ref{thm:rm-unisolvence} then determines a unique
$v\in P^+$ with these prescribed values.
Using \eqref{eq:face-sum-edge-decomposition}, the prescribed values determine
all the remaining $(d-2)$-face vertex sums by the same propagation argument
as in Lemma~\ref{lem:rm-face-sums}. Hence, by
\eqref{eq:a-compatibility},
\[
\Sigma_{E^{i,k;j,l}}(v)=a^{i,k;j,l},
\qquad
1\leq i\neq j\leq d,\quad k,l\in\{1,2\}.
\]
Hence $v$ satisfies all the prescribed data in \eqref{eq:p2-data}.

It remains to show that $v\in P_2(K)$. Suppose
\[
v=v_0+\sum_{s=1}^d\zeta_s x_s^3,
\qquad v_0\in P_2(K).
\]
Note that the data of $v$ satisfy \eqref{eq:b-a-compatibility}, 
Lemma~\ref{lem:p2-b-a-compatibility} shows that  $\sum_{s=1}^d\zeta_sx_s^3$ satisfies \eqref{eq:assum}. For each $1\leq i\leq d$,
substitution of $\sum_{s=1}^d\zeta_sx_s^3$ into
\eqref{eq:assum} yields
\[
2^{d-2}\zeta_i
\left[
2(x_{i,1}^2+x_{i,1}x_{i,2}+x_{i,2}^2)
-3(x_{i,1}^2+x_{i,2}^2)
\right]
=
-2^{d-2}\zeta_i(x_{i,2}-x_{i,1})^2=0.
\]
Since $x_{i,1}\ne x_{i,2}$, it follows that
$\zeta_i=0$ for all $1\leq i\leq d$. Thus $v\in P_2(K)$, which proves
existence.

For uniqueness, let $v_1,v_2\in P_2(K)$ satisfy
\eqref{eq:p2-data}. Their difference has zero $(d-2)$-face vertex sums and
zero coordinate-derivative average values. Hence
\[
\mathcal N^+(v_1-v_2)=0.
\]
Theorem~\ref{thm:rm-unisolvence} gives $v_1=v_2$. This concludes the proof.
\end{proof}

\begin{remark}
\label{rem:cellwise-compatibility}
Recall that $e_i$ is the $i$-th coordinate vector in $\mathbb{R}^d$.  For an element
$K_\alpha$, denote its lower and upper $x_i$-faces by
$F_\alpha^i$ and $F_{\alpha+e_i}^i$, respectively. For
$1\le i\ne j\le d$, let $a_\gamma^{i,j}$ denote the vertex-sum value
on $E_\gamma^{i,j}$ and $b_\beta^i$ the coordinate-derivative average value on $F_\beta^i$, with
$a_\gamma^{i,j}=a_\gamma^{j,i}$.

In this notation, \eqref{eq:a-compatibility} reads
\begin{equation}
\label{eq:cellwise-a-compatibility}
a_{\alpha+(k-1)e_i}^{i,j}
+a_{\alpha+(k-1)e_i+e_j}^{i,j}
=
a_{\alpha+(k-1)e_i}^{i,q}
+a_{\alpha+(k-1)e_i+e_q}^{i,q},
\end{equation}
for $1\le i\le d$, $k\in\{1,2\}$, and $j,q\ne i$. Moreover,
\eqref{eq:b-a-compatibility} reads
\begin{equation}
\label{eq:cellwise-b-a-compatibility}
\frac{
a_{\alpha+e_i}^{i,j}
+a_{\alpha+e_i+e_j}^{i,j}
-a_\alpha^{i,j}
-a_{\alpha+e_j}^{i,j}
}{
h_{i,\alpha_i}
}
=
2^{d-2}\bigl(b_\alpha^i+b_{\alpha+e_i}^i\bigr),
\end{equation}
for $1\le i\le d$ and any $j\ne i$.
\end{remark}
\subsection{Patch basis functions}\label{sec:basis-vh0}
This subsection constructs the local patch functions used in the global basis. For each interior element, a $3^d$-element patch and the corresponding homogeneous nonconforming $P_2$ space are introduced. The patch space is shown to be one-dimensional, and its nonzero basis function is then extended by zero to $\Omega$.

For each interior index $\alpha\in\IKhint$, let $\mathcal{M}_{K_\alpha}$ denote the $3^d$-element patch centered at $K_\alpha$, and let $V_{h0}(\mathcal{M}_{K_\alpha})$ be the homogeneous nonconforming $P_2$ space on this patch, with the
homogeneous boundary conditions imposed on
$\partial\mathcal{M}_{K_\alpha}$ as in
\eqref{ncp2-space-homogeneous}. See Figure~\ref{fig:three-dimensional-situation-one}(a) below for a patch in three dimensions.

\begin{lemma}
\label{lem:single-interior-zero}
For every $\alpha\in\IKhint$ and any
$v\in V_{h0}(\mathcal{M}_{K_\alpha})$, if
\[
\Sigma_{
E_{\alpha-\sum_{r=3}^{d}e_r}^{1,2}
}(v)=0,
\]
then $v=0$ on $\mathcal{M}_{K_\alpha}$. When $d=2$, the sum
$\sum_{r=3}^{d}e_r$ is empty and hence equals zero, so the condition
becomes
$\Sigma_{E_\alpha^{1,2}}(v)
=v(x_{1,\alpha_1},x_{2,\alpha_2})=0$.
\end{lemma}
\begin{proof}
It suffices to consider
$ 
\mathcal M:=\mathcal M_{K_{(2,2,\ldots,2)}},
$
since the general case follows by a change of global indices. The
prescribed condition becomes
$
\Sigma_{E_{\gamma}^{1,2}}(v)=0.
$ with $\gamma=(2,2,1,\ldots,1)$.

\textit{Step 1 considers the lower-corner element.}
For $K_{(1,\ldots,1)}$, the homogeneous boundary conditions give
zero vertex sums and coordinate-derivative averages on the boundary faces of
$\mathcal M$. Together with
$
\Sigma_{E_{\gamma}^{1,2}}(v)=0
$, 
the compatibility conditions
\eqref{eq:a-compatibility} and
\eqref{eq:b-a-compatibility} imply that all the degrees of freedom
of $K_{(1,\ldots,1)}$ vanish. Therefore, Theorem~\ref{lem:p2-data-compatibility} gives
$
v|_{K_{(1,\ldots,1)}}=0.
$

\textit{Step 2 presents a two-element propagation argument.}
Consider
$ 
K^-:=K_{(1,\ldots,1,2)}$ and $
K^+:=K_{(1,\ldots,1,3)},
$ 
with common face
$
F_*:=K^-\cap K^+=F_{(1,\ldots,1,3)}^d 
$. 
Step~1, the homogeneous boundary conditions, and \eqref{eq:a-compatibility} show that all the
unfixed $(d-2)$-face vertex sums on $F_*$ have a common value
$\sigma$. Denote
\[
g:=\fint_{F_*}\frac{\partial v}{\partial x_d}\,ds.
\]
Applying \eqref{eq:b-a-compatibility} to $K^-$ and $K^+$ gives
\[
\frac{\sigma}{h_{d,2}}=2^{d-2}g,
\qquad
-\frac{\sigma}{h_{d,3}}=2^{d-2}g.
\]
This gives $\sigma=g=0$. Together with the previously known vanishing data, all $(d-2)$-face vertex sums and coordinate-derivative averages on both $K^-$ and $K^+$ vanish. Therefore, Theorem~\ref{lem:p2-data-compatibility} gives
\[
v|_{K^-}=v|_{K^+}=0.
\]
Applying the same argument in each coordinate direction gives
\begin{equation}
\label{eq:patch-base-zero}
v|_{K_\beta}=0
\qquad\text{whenever}\qquad
\#\{r:\beta_r=2\}\leq1.
\end{equation}
This two-element propagation argument is referred to as \textbf{Situation~1} below.
For $d=3$, this propagation argument is illustrated in
Figure~\ref{fig:three-dimensional-situation-one}.

\begin{figure}[H]
\centering
\begin{tikzpicture}[
  zeroedge/.style={red!75!black,very thick},
  sigmaedge/.style={blue!75!black,very thick},
  hiddenline/.style={black!65,thin,dashed},
  gridline/.style={black!55,thin},
  outeredge/.style={black,semithick},
  axis/.style={->,>=stealth,black,semithick},
  flow/.style={->,>=stealth,semithick}
]
\begin{scope}[
  shift={(2.55,6.15)},
  x={(0.48cm,-0.20cm)},
  y={(-0.48cm,-0.20cm)},
  z={(0cm,0.58cm)}
]
  \foreach \a in {0,1,2,3}
    \foreach \b in {0,1,2,3}{
      \draw[hiddenline] (0,\a,\b)--(3,\a,\b);
      \draw[hiddenline] (\a,0,\b)--(\a,3,\b);
      \draw[hiddenline] (\a,\b,0)--(\a,\b,3);
    }
  \draw[outeredge]
    (0,0,0)--(3,0,0) (0,3,0)--(3,3,0)
    (0,0,3)--(3,0,3) (0,3,3)--(3,3,3);
  \draw[outeredge]
    (0,0,0)--(0,3,0) (3,0,0)--(3,3,0)
    (0,0,3)--(0,3,3) (3,0,3)--(3,3,3);
  \draw[outeredge]
    (0,0,0)--(0,0,3) (3,0,0)--(3,0,3)
    (0,3,0)--(0,3,3) (3,3,0)--(3,3,3);
  \draw[axis] (0,0,0)--(3.35,0,0)
    node[font=\scriptsize,below right,inner sep=1pt] {$y$};
  \draw[axis] (0,0,0)--(0,3.35,0)
    node[font=\scriptsize,below left,inner sep=1pt] {$x$};
  \draw[axis] (0,0,0)--(0,0,3.35)
    node[font=\scriptsize,above left,inner sep=1pt] {$z$};
  \node[font=\tiny,fill=white,inner sep=0.4pt]
        at (0.50,0.50,0.50)
    {$K_{(1,1,1)}$};
  \node[font=\tiny,fill=white,inner sep=0.4pt]
        at (0.50,0.50,1.50)
    {$K_{(1,1,2)}$};
  \node[font=\tiny,fill=white,inner sep=0.4pt]
        at (0.50,0.50,2.50)
    {$K_{(1,1,3)}$};
\end{scope}
\node[font=\small,align=center] at (2.55,4.48)
  {(a) the $3^3$-element patch\\[-1pt]
   $\mathcal M_{K_{(2,2,2)}}$};

\begin{scope}[
  shift={(8.05,6.55)},
  x={(0.82cm,-0.31cm)},
  y={(-0.82cm,-0.31cm)},
  z={(0cm,0.94cm)}
]
  \coordinate (A)  at (0,0,0);
  \coordinate (B)  at (1,0,0);
  \coordinate (C)  at (0,1,0);
  \coordinate (D)  at (1,1,0);
  \coordinate (At) at (0,0,1);
  \coordinate (Bt) at (1,0,1);
  \coordinate (Ct) at (0,1,1);
  \coordinate (Dt) at (1,1,1);
  \draw[zeroedge] (A)--(B)--(D)--(C)--cycle;
  \draw[zeroedge] (A)--(At) (B)--(Bt) (C)--(Ct);
  \draw[zeroedge] (At)--(Bt) (At)--(Ct);
  \draw[sigmaedge] (D)--(Dt) (Bt)--(Dt) (Ct)--(Dt);
\end{scope}
\node[font=\scriptsize,align=center] at (8.05,5.20)
  {\textcolor{red!75!black}{red: boundary edges, $\Sigma_E=0$}\\[-1pt]
   \textcolor{blue!75!black}{blue: the three edge sums are equal;}\\[-1pt]
   \textcolor{blue!75!black}{since $\Sigma_{E_{(2,2,1)}^{1,2}}(v)=0$, they vanish}};
\node[font=\small] at (8.05,4.55)
  {(b) the lower-corner element $K_{(1,1,1)}$};

\begin{scope}[
  shift={(2.35,1.45)},
  x={(0.76cm,-0.27cm)},
  y={(-0.76cm,-0.27cm)},
  z={(0cm,0.86cm)}
]
  \coordinate (A0) at (0,0,0);
  \coordinate (B0) at (1,0,0);
  \coordinate (C0) at (0,1,0);
  \coordinate (D0) at (1,1,0);
  \coordinate (A1) at (0,0,1);
  \coordinate (B1) at (1,0,1);
  \coordinate (C1) at (0,1,1);
  \coordinate (D1) at (1,1,1);
  \coordinate (A2) at (0,0,2);
  \coordinate (B2) at (1,0,2);
  \coordinate (C2) at (0,1,2);
  \coordinate (D2) at (1,1,2);
  \draw[zeroedge] (A0)--(B0)--(D0)--(C0)--cycle;
  \draw[zeroedge] (A0)--(A1) (B0)--(B1) (C0)--(C1);
  \draw[sigmaedge] (D0)--(D1);
  \draw[zeroedge] (A1)--(B1) (A1)--(C1);
  \draw[sigmaedge] (B1)--(D1) (C1)--(D1);
  \draw[zeroedge] (A1)--(A2) (B1)--(B2) (C1)--(C2);
  \draw[sigmaedge] (D1)--(D2);
  \draw[zeroedge] (A2)--(B2)--(D2)--(C2)--cycle;
\end{scope}
\node[font=\scriptsize,anchor=east] at (1.20,2.88) {$K_{(1,1,3)}$};
\node[font=\scriptsize,anchor=east] at (1.20,1.82) {$K_{(1,1,2)}$};
\node[font=\scriptsize,align=center] at (2.35,0.54)
  {\textcolor{red!75!black}{red: known $\Sigma_E=0$}\\[-1pt]
   \textcolor{blue!75!black}{blue: equal edge sums $\Sigma_E=\sigma$}};
\node[font=\small,align=center] at (2.35,-0.32)
  {(c) edge vertex sums of\\[-1pt]
   $K_{(1,1,2)}$ and $K_{(1,1,3)}$};

\begin{scope}[
  shift={(6.55,1.45)},
  x={(0.76cm,-0.27cm)},
  y={(-0.76cm,-0.27cm)},
  z={(0cm,0.86cm)}
]
  \coordinate (A0) at (0,0,0);
  \coordinate (B0) at (1,0,0);
  \coordinate (C0) at (0,1,0);
  \coordinate (D0) at (1,1,0);
  \coordinate (A1) at (0,0,1);
  \coordinate (B1) at (1,0,1);
  \coordinate (C1) at (0,1,1);
  \coordinate (D1) at (1,1,1);
  \coordinate (A2) at (0,0,2);
  \coordinate (B2) at (1,0,2);
  \coordinate (C2) at (0,1,2);
  \coordinate (D2) at (1,1,2);

  \fill[yellow!16] (A1)--(B1)--(D1)--(C1)--cycle;
  \draw[outeredge] (A0)--(B0)--(D0)--(C0)--cycle;
  \draw[outeredge] (A0)--(A1) (B0)--(B1) (C0)--(C1) (D0)--(D1);
  \draw[outeredge] (A1)--(A2) (B1)--(B2) (C1)--(C2) (D1)--(D2);
  \draw[outeredge] (A2)--(B2)--(D2)--(C2)--cycle;
  \draw[blue!65!black,thick,dashed]
    (A1)--(B1)--(D1)--(C1)--cycle;
  \node[font=\scriptsize,fill=yellow!16,inner sep=0.4pt]
    at (0.23,0.23,1) {$F_*$};
\end{scope}
\node[font=\scriptsize,align=left,anchor=west] at (7.75,2.94)
  {$K_{(1,1,3)}:\quad\displaystyle-\frac{\sigma}{h_{3,3}}=2g$};
\node[font=\scriptsize,align=left,anchor=west] at (7.75,2.24)
  {$\displaystyle g=\fint_{F_*}\frac{\partial v}{\partial x_3}\,ds$};
\node[font=\scriptsize,align=left,anchor=west] at (7.75,1.54)
  {$K_{(1,1,2)}:\quad\displaystyle\frac{\sigma}{h_{3,2}}=2g$};
\node[font=\small,align=center] at (7.65,-0.32)
  {(d) the common-face relations of\\[-1pt]
   $K_{(1,1,2)}$ and $K_{(1,1,3)}$};
\end{tikzpicture}
\caption{Illustration of the two-element propagation argument in Lemma~\ref{lem:single-interior-zero} for $d=3$.}
\label{fig:three-dimensional-situation-one}
\end{figure}

\textit{Step 3 completes the proof by induction.}
Assume that, for some $1\leq m<d$,
\[
v|_{K_\beta}=0
\qquad
\forall\,\beta\in\{1,2,3\}^d
\quad\text{with}\quad
\#\{r:\beta_r=2\}\leq m.
\]
Let $\alpha\in\{1,2,3\}^d$ satisfy
\[
\#\{r:\alpha_r=2\}=m+1.
\]

For every coordinate direction $j$, at least one of the two $x_j$-faces
of $K_\alpha$ has zero data. More precisely,
\[
\begin{cases}
F_\alpha^j\subset\partial\mathcal M,
& \alpha_j=1,\\[1mm]
v|_{K_{\alpha-e_j}}
=
v|_{K_{\alpha+e_j}}
=0,
& \alpha_j=2,\\[1mm]
F_{\alpha+e_j}^j\subset\partial\mathcal M,
& \alpha_j=3.
\end{cases}
\]
In the case $\alpha_j=2$, both neighboring elements satisfy the induction
hypothesis, since replacing $\alpha_j=2$ by either $1$ or $3$ reduces
the number of components equal to $2$ from $m+1$ to $m$. Thus both
$x_j$-faces have zero data whenever $\alpha_j=2$.

Since $\#\{r:\alpha_r=2\}=m+1\geq1$, there exists at least one
coordinate direction $i$ with $\alpha_i=2$. Hence both $x_i$-faces
of $K_\alpha$ have zero data, while in every coordinate direction at
least one of the two faces has zero data.

By \eqref{eq:a-compatibility} and the same propagation argument as in
Lemma~\ref{lem:rm-face-sums}, all $(d-2)$-face vertex sums on
$K_\alpha$ vanish. Since one coordinate-derivative average in each
direction is already zero, \eqref{eq:b-a-compatibility} implies that the
average on the opposite face also vanishes. Thus all degrees of freedom
of $v|_{K_\alpha}$ vanish, and Theorem~\ref{thm:rm-unisolvence} gives
$
v|_{K_\alpha}=0 
$.
Therefore the assertion also holds for all indices with
$
\#\{r:\alpha_r=2\}=m+1
$.

Induction yields
\[
v|_{K_\alpha}=0
\qquad
\forall\,\alpha\in\{1,2,3\}^d,
\]
and hence $v=0$ on $\mathcal M$.
\end{proof}
\begin{remark}\label{rem:no-interior-elements}
\textbf{Situation~1} immediately yields the following consequence. If a
rectangular mesh $\Kh$ of $\Omega$ has no interior elements,
equivalently, if $N_i\leq2$ for some $1\leq i\leq d$, then
$V_{h0}=\{0\}$.
\end{remark}

For each $\alpha\in\IKhint$, the patch function $\varphi_\alpha$ is
defined by prescribing its vertex-sum values on $(d-2)$-faces and
coordinate-derivative averages on $(d-1)$-faces of
$\mathcal{M}_{K_\alpha}$ as follows.
\begin{df}
\label{def:patch-function}
For
$1\leq m\leq d$, let
\[
\omega_{\alpha;m,\alpha_m-1}
:=
\frac{h_{m,\alpha_m-1}}
{h_{m,\alpha_m-1}+h_{m,\alpha_m}},
\qquad
\omega_{\alpha;m,\alpha_m+1}
:=
\frac{h_{m,\alpha_m+1}}
{h_{m,\alpha_m+1}+h_{m,\alpha_m}},
\]
and
\[
\omega_{\alpha;m,\alpha_m}
:=
\omega_{\alpha;m,\alpha_m-1}
+\omega_{\alpha;m,\alpha_m+1}.
\]
For $r\in\{\alpha_m,\alpha_m+1\}$, let
\[
\eta_{\alpha,m}(r)
:=
\begin{cases}
\alpha_m-1,&r=\alpha_m,\\
\alpha_m+1,&r=\alpha_m+1.
\end{cases}
\]
Let $1\leq i\ne j\leq d$. On every interior $(d-2)$-face
$E_\gamma^{i,j}$ of $\mathcal{M}_{K_\alpha}$, let
\begin{equation}
\label{eq:general-patch-vertex-sum}
\Sigma_{E_\gamma^{i,j}}(\varphi_\alpha)
=A_\gamma^{i,j}
:=
\omega_{\alpha;i,\eta_{\alpha,i}(\gamma_i)}
\omega_{\alpha;j,\eta_{\alpha,j}(\gamma_j)}
\prod_{m\notin\{i,j\}}\omega_{\alpha;m,\gamma_m},
\end{equation}
where
\[
\gamma_i\in\{\alpha_i,\alpha_i+1\},\qquad
\gamma_j\in\{\alpha_j,\alpha_j+1\},\qquad
\gamma_m\in\{\alpha_m-1,\alpha_m,\alpha_m+1\}
\quad(m\notin\{i,j\}).
\]
On each $(d-2)$-face contained in
$\partial\mathcal{M}_{K_\alpha}$, prescribe
$A_\gamma^{i,j}=0$.
For each interior $(d-1)$-face $F_\beta^i$ of
$\mathcal{M}_{K_\alpha}$, prescribe
\begin{equation}
\label{eq:general-patch-derivative-average}
\fint_{F_\beta^i}
\frac{\partial\varphi_\alpha}{\partial x_i}\,ds
=D_i(\beta)
:=
\begin{cases}
\displaystyle
\frac{
\omega_{\alpha;i,\alpha_i-1}
\prod_{m\neq i}\omega_{\alpha;m,\beta_m}
}{
2^{d-2}h_{i,\alpha_i-1}
},
&\beta_i=\alpha_i,\\[12pt]
\displaystyle
-\frac{
\omega_{\alpha;i,\alpha_i+1}
\prod_{m\neq i}\omega_{\alpha;m,\beta_m}
}{
2^{d-2}h_{i,\alpha_i+1}
},
&\beta_i=\alpha_i+1,
\end{cases}
\end{equation}
where
\[
\beta_m\in\{\alpha_m-1,\alpha_m,\alpha_m+1\}
\qquad(m\neq i).
\]
On every $(d-1)$-face contained in
$\partial\mathcal{M}_{K_\alpha}$, prescribe $D_i(\beta)=0$.
\end{df}

\begin{remark}
For the same $\gamma$ and any $i\ne j$, the previous definition gives
\[
A_\gamma^{i,j}=A_\gamma^{j,i}.
\]
Hence the two symbolic representations $E_\gamma^{i,j}$ and
$E_\gamma^{j,i}$ of the same geometric $(d-2)$-face receive the same
prescribed value, so each$(d-2)$-face has a unique
prescribed value.  For $d=3$, the prescribed values on the six interior coordinate
sections of the patch centered at $K_{(2,2,2)}$ are illustrated in
Figure~\ref{fig:three-dimensional-patch-prescriptions}.
For compactness in this figure, write
\(
W_{r_1,r_2,r_3}
:=
\prod_{m=1}^{3}\omega_{(2,2,2);m,r_m}.
\)
\end{remark}
\tikzset{
  sectionboundary/.style={red!75!black,semithick},
  sectioninterior/.style={blue!75!black,semithick},
  sectiongrid/.style={black!55,thin},
  sectiontitle/.style={font=\scriptsize},
  sectionkind/.style={font=\tiny},
  sectionvalue/.style={
    font=\fontsize{6.5}{6.7}\selectfont,
    align=center,
    inner sep=0.2pt
  },
  sectionedgevalue/.style={
    font=\fontsize{6.5}{6.7}\selectfont,
    align=center,
    fill=white,
    fill opacity=0.94,
    text opacity=1,
    inner sep=0.15pt
  },
  sectionzero/.style={
    font=\scriptsize,
    text=red!75!black,
    fill=white,
    inner sep=0.1pt
  },
  sectionaxis/.style={->,red!75!black,thin},
  sectionaxislabel/.style={
    font=\scriptsize,
    fill=white,
    inner sep=0.2pt
  }
}

\newcommand{\verticalW}[4]{%
  \ifcase#1\relax
  \or W_{#2,#3,#4}%
  \or W_{#3,#2,#4}%
  \or W_{#3,#4,#2}%
  \fi
}

\newcommand{\horizontalW}[4]{%
  \ifcase#1\relax
  \or W_{#2,#3,#4}%
  \or W_{#3,#2,#4}%
  \or W_{#3,#4,#2}%
  \fi
}

\newcommand{\faceW}[4]{%
  \ifcase#1\relax
  \or W_{#2,#3,#4}%
  \or W_{#3,#2,#4}%
  \or W_{#3,#4,#2}%
  \fi
}

\newcommand{\patchedgepanel}[6]{%
  \begin{scope}[shift={(#1,#2)},scale=1.05]
    \draw[sectionboundary] (0,0) rectangle (3,3);
    \draw[sectioninterior]
      (1,0)--(1,3)
      (2,0)--(2,3)
      (0,1)--(3,1)
      (0,2)--(3,2);

    \foreach \s in {1,2,3}{
      \node[sectionedgevalue,rotate=90] at (1,{\s-0.5})
        {$\verticalW{#3}{#4}{1}{\s}$};
      \node[sectionedgevalue,rotate=90] at (2,{\s-0.5})
        {$\verticalW{#3}{#4}{3}{\s}$};
    }

    \foreach \r in {1,2,3}{
      \node[sectionedgevalue] at ({\r-0.5},1)
        {$\horizontalW{#3}{#4}{\r}{1}$};
      \node[sectionedgevalue] at ({\r-0.5},2)
        {$\horizontalW{#3}{#4}{\r}{3}$};
    }

    \foreach \t in {0.5,1.5,2.5}{
      \node[sectionzero,below] at (\t,0) {$0$};
      \node[sectionzero,above] at (\t,3) {$0$};
      \node[sectionzero,left]  at (0,\t) {$0$};
      \node[sectionzero,right] at (3,\t) {$0$};
    }

    \node[sectionkind] at (1.5,-0.28) {edge vertex sums};
  \end{scope}%
}

\newcommand{\patchfacepanel}[7]{%
  \begin{scope}[shift={(#1,#2)},scale=1.05]
    \draw[sectiongrid] (0,0) rectangle (3,3);
    \draw[sectiongrid]
      (1,0)--(1,3)
      (2,0)--(2,3)
      (0,1)--(3,1)
      (0,2)--(3,2);

    \foreach \r in {1,2,3}{
      \foreach \s in {1,2,3}{
        \node[sectionvalue] at ({\r-0.5},{\s-0.5})
          {$#7\tfrac{\faceW{#3}{#4}{\r}{\s}}{2h_{#3,#4}}$};
      }
    }

    \node[sectionkind] at (1.5,-0.28) {coordinate-derivative averages};
  \end{scope}%
}

\newcommand{\patchsectionaxes}[4]{%
  \begin{scope}[shift={(#1,#2)}]
    \draw[sectionaxis] (0,0)--(3.12,0)
      node[sectionaxislabel,right] {$x_{#3}$};
    \draw[sectionaxis] (0,0)--(0,3.12)
      node[sectionaxislabel,above] {$x_{#4}$};
  \end{scope}%
}

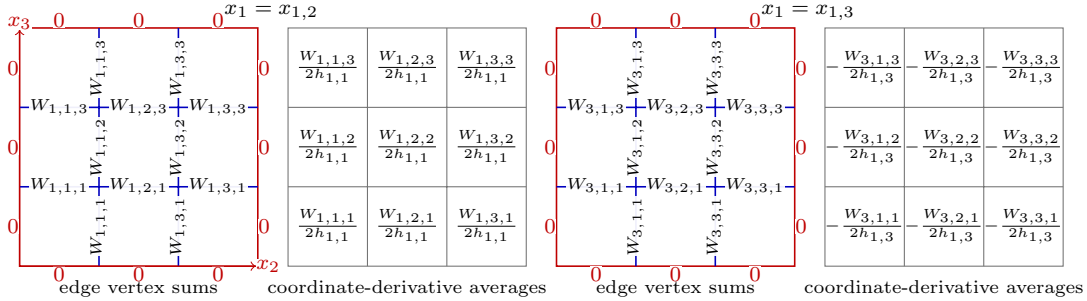
\begin{figure}[H]
\centering
\captionsetup{aboveskip=4pt,belowskip=0pt}
\begin{tikzpicture}
  \node[sectiontitle] at (3.33,3.50) {$x_1=x_{1,2}$};
  \node[sectiontitle] at (10.43,3.50) {$x_1=x_{1,3}$};

  \patchedgepanel{0.00}{0.15}{1}{1}{2}{3}
  \patchsectionaxes{0.00}{0.15}{2}{3}
  \patchfacepanel{3.55}{0.15}{1}{1}{2}{3}{}

  \patchedgepanel{7.10}{0.15}{1}{3}{2}{3}
  \patchfacepanel{10.65}{0.15}{1}{3}{2}{3}{-}
\end{tikzpicture}
\caption{Prescribed vertex sums and coordinate-derivative averages on the $x_1$-sections of $\mathcal M_{K_{(2,2,2)}}$ for $d=3$.}
\label{fig:three-dimensional-patch-prescriptions}
\end{figure}

\begin{figure}[t]
\ContinuedFloat
\centering
\captionsetup{aboveskip=4pt,belowskip=0pt}
\begin{tikzpicture}
  \node[sectiontitle] at (3.33,5.10) {$x_2=x_{2,2}$};
  \node[sectiontitle] at (10.43,5.10) {$x_2=x_{2,3}$};

  \patchedgepanel{0.00}{1.75}{2}{1}{1}{3}
  \patchsectionaxes{0.00}{1.75}{1}{3}
  \patchfacepanel{3.55}{1.75}{2}{1}{1}{3}{}

  \patchedgepanel{7.10}{1.75}{2}{3}{1}{3}
  \patchfacepanel{10.65}{1.75}{2}{3}{1}{3}{-}

  \node[sectiontitle] at (3.33,0.60) {$x_3=x_{3,2}$};
  \node[sectiontitle] at (10.43,0.60) {$x_3=x_{3,3}$};

  \patchedgepanel{0.00}{-2.70}{3}{1}{1}{2}
  \patchsectionaxes{0.00}{-2.70}{1}{2}
  \patchfacepanel{3.55}{-2.70}{3}{1}{1}{2}{}

  \patchedgepanel{7.10}{-2.70}{3}{3}{1}{2}
  \patchfacepanel{10.65}{-2.70}{3}{3}{1}{2}{-}
\end{tikzpicture}
\caption[]{Continued: prescribed values on the $x_2$- and $x_3$-sections.}
\end{figure}
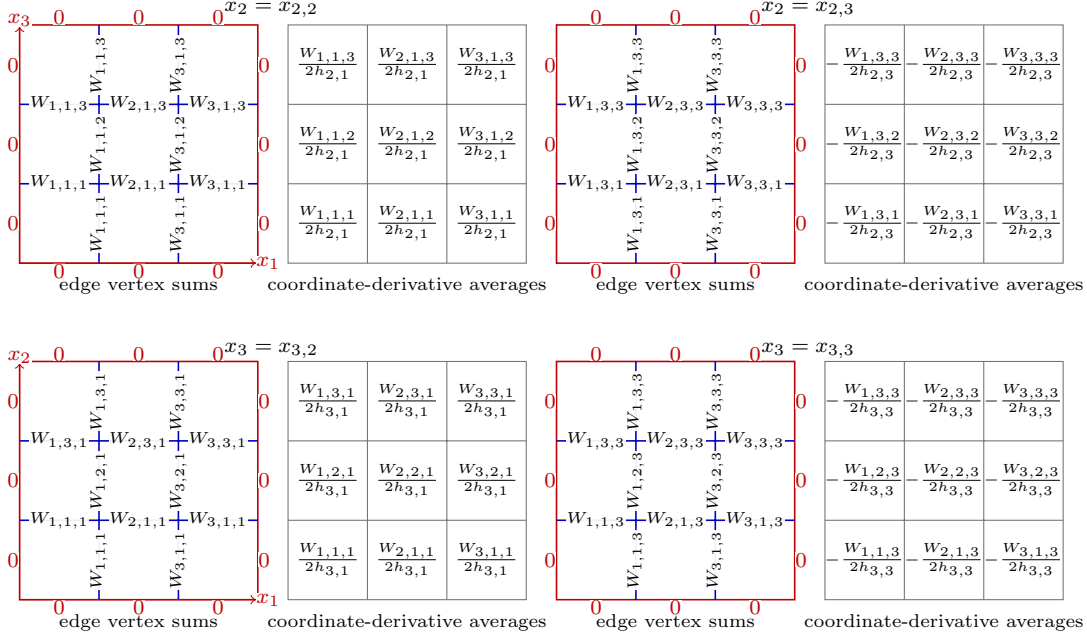

The following theorem proves that the prescriptions in Definition \ref{def:patch-function} are
compatible and uniquely determine a function in
$V_{h0}(\mathcal{M}_{K_\alpha})$.
\begin{theorem}
\label{lem:single-interior-dimension}
For every $\alpha\in\IKhint$,
\[
\dim V_{h0}(\mathcal{M}_{K_\alpha})=1.
\]
\end{theorem}

\begin{proof}
Lemma~\ref{lem:single-interior-zero} implies
\[
\dim V_{h0}(\mathcal{M}_{K_\alpha})\leq1.
\]
It remains to show that the prescriptions in
Definition~\ref{def:patch-function} define a nonzero function
$\varphi_\alpha\in V_{h0}(\mathcal{M}_{K_\alpha})$.

Since each   $(d-2)$-face and $(d-1)$-face is assigned a
single value, the interior continuity conditions are automatically
satisfied. The prescribed values also vanish on
$\partial\mathcal{M}_{K_\alpha}$. Thus, by
Theorem~\ref{lem:p2-data-compatibility} and
Remark~\ref{rem:cellwise-compatibility}, it suffices to verify
\eqref{eq:cellwise-a-compatibility} and
\eqref{eq:cellwise-b-a-compatibility} on each element of the patch.

\textit{Verification of \eqref{eq:cellwise-a-compatibility}.}
Fix an interior $(d-1)$-face $F_\beta^i$ and $j\ne i$, and recall the
assigned vertex-sum values on its two $(d-2)$-faces in the
$x_j$-direction by $A_\beta^{i,j}$ and
$A_{\beta+e_j}^{i,j}$. By
\eqref{eq:general-patch-vertex-sum}, both values have the common factor
\[
\omega_{\alpha;i,\eta_{\alpha,i}(\beta_i)}
\prod_{m\notin\{i,j\}}\omega_{\alpha;m,\beta_m}.
\]
If $\beta_j=\alpha_j-1$, then $A_\beta^{i,j}=0$ and
$A_{\beta+e_j}^{i,j}$ contains the factor
$\omega_{\alpha;j,\alpha_j-1}$. If $\beta_j=\alpha_j$, both
$(d-2)$-faces are interior, and the sum of their corresponding factors is
\[
\omega_{\alpha;j,\alpha_j-1}
+\omega_{\alpha;j,\alpha_j+1}
=
\omega_{\alpha;j,\alpha_j}.
\]
If $\beta_j=\alpha_j+1$, then $A_{\beta+e_j}^{i,j}=0$ and
$A_\beta^{i,j}$ contains the factor
$\omega_{\alpha;j,\alpha_j+1}$. Therefore, in all cases,
\[
A_\beta^{i,j}+A_{\beta+e_j}^{i,j}
=
\omega_{\alpha;i,\eta_{\alpha,i}(\beta_i)}
\prod_{m\neq i}\omega_{\alpha;m,\beta_m}.
\]
The right-hand side is independent of $j$, and hence
\eqref{eq:cellwise-a-compatibility} holds. On a boundary
$(d-1)$-face, it holds trivially since all the corresponding
vertex-sum values vanish. 
The right-hand side is independent
of $j$. Hence \eqref{eq:cellwise-a-compatibility} holds on every
interior $(d-1)$-face; on a boundary face it holds trivially since
all the corresponding values vanish.

For later use, let $ 
S_\beta^i
:=
A_\beta^{i,j}+A_{\beta+e_j}^{i,j}
$ for $j\neq i$.
The previous argument shows that $S_\beta^i$ is well defined and,
for an interior face,
\begin{equation}
\label{eq:patch-face-sum}
S_\beta^i
=
\omega_{\alpha;i,\eta_{\alpha,i}(\beta_i)}
\prod_{m\neq i}\omega_{\alpha;m,\beta_m},
\end{equation}
while $S_\beta^i=0$ on $\partial\mathcal{M}_{K_\alpha}$.

\textit{Verification of \eqref{eq:cellwise-b-a-compatibility}.}
Let $K_\kappa\subset\mathcal{M}_{K_\alpha}$, where
$\kappa_m\in\{\alpha_m-1,\alpha_m,\alpha_m+1\}$. For $j\ne i$, the
left-hand side of \eqref{eq:cellwise-b-a-compatibility} can be written
as
$
\frac{S_{\kappa+e_i}^i-S_\kappa^i}{h_{i,\kappa_i}}
$. 
Let
\[
W_\kappa^i
:=
\prod_{m\neq i}\omega_{\alpha;m,\kappa_m}.
\]
Using \eqref{eq:patch-face-sum} and the vanishing boundary values gives
\[
\frac{S_{\kappa+e_i}^i-S_\kappa^i}{h_{i,\kappa_i}}
=
W_\kappa^i
\begin{cases}
\displaystyle
\frac{\omega_{\alpha;i,\alpha_i-1}}
     {h_{i,\alpha_i-1}},
&\kappa_i=\alpha_i-1,\\[9pt]
\displaystyle
\frac{
\omega_{\alpha;i,\alpha_i+1}
-\omega_{\alpha;i,\alpha_i-1}
}{h_{i,\alpha_i}},
&\kappa_i=\alpha_i,\\[9pt]
\displaystyle
-\frac{\omega_{\alpha;i,\alpha_i+1}}
      {h_{i,\alpha_i+1}},
&\kappa_i=\alpha_i+1.
\end{cases}
\]
On the other hand, the prescribed coordinate-derivative average values give
\[
2^{d-2}\bigl(D_i(\kappa)+D_i(\kappa+e_i)\bigr)
=
W_\kappa^i
\begin{cases}
\displaystyle
\frac{\omega_{\alpha;i,\alpha_i-1}}
     {h_{i,\alpha_i-1}},
&\kappa_i=\alpha_i-1,\\[9pt]
\displaystyle
\frac{\omega_{\alpha;i,\alpha_i-1}}
     {h_{i,\alpha_i-1}}
-\frac{\omega_{\alpha;i,\alpha_i+1}}
      {h_{i,\alpha_i+1}},
&\kappa_i=\alpha_i,\\[9pt]
\displaystyle
-\frac{\omega_{\alpha;i,\alpha_i+1}}
      {h_{i,\alpha_i+1}},
&\kappa_i=\alpha_i+1.
\end{cases}
\]
The two expressions agree in the first and third cases. For
$\kappa_i=\alpha_i$, the equality follows from
\[
\frac{
\omega_{\alpha;i,\alpha_i+1}
-\omega_{\alpha;i,\alpha_i-1}
}{h_{i,\alpha_i}}
=
\frac{\omega_{\alpha;i,\alpha_i-1}}
     {h_{i,\alpha_i-1}}
-
\frac{\omega_{\alpha;i,\alpha_i+1}}
     {h_{i,\alpha_i+1}},
\]
which follows directly from the definition of the weights.
Therefore \eqref{eq:cellwise-b-a-compatibility} holds on every
element of the patch.

Theorem~\ref{lem:p2-data-compatibility} gives a unique polynomial in $P_2(K_\kappa)$ on each element, and these elementwise polynomials define a function $\varphi_\alpha\in V_{h0}(\mathcal{M}_{K_\alpha})$. 
Moreover,
\[
\Sigma_{
E_{\alpha-\sum_{r=3}^{d}e_r}^{1,2}
}(\varphi_\alpha)
=
\prod_{m=1}^{d}
\omega_{\alpha;m,\alpha_m-1}
>0.
\]
Hence $\varphi_\alpha\neq0$, and therefore
\[
\dim V_{h0}(\mathcal{M}_{K_\alpha})\geq1.
\]
This concludes the proof.
\end{proof}
In the remainder of the paper, $\varphi_\alpha$ also denotes its
extension by zero from $\mathcal{M}_{K_\alpha}$ to $\Omega$.
\subsection{Proof of Theorem~\ref{thm:dim-theory}}\label{sec:proof-dim-theory}
\begin{proof}
For every $\alpha\in\IKhint$, Theorem~\ref{lem:single-interior-dimension}
shows that
$V_{h0}(\mathcal{M}_{K_\alpha})
=\operatorname{span}\{\varphi_\alpha\}$.
Extending $\varphi_\alpha$ by zero outside
$\mathcal{M}_{K_\alpha}$ gives a function in $V_{h0}$.

It remains to prove that these functions span $V_{h0}$ and are
linearly independent. Let $\mathbf 1=(1,\ldots,1)$ and order the multi-indices $\kappa=(\kappa_1,\ldots,\kappa_d)$, with $1\leq\kappa_i\leq N_i-2$, lexicographically. For each such $\kappa$,
let $\alpha=\kappa+\mathbf1\in\IKhint$. Then $K_\kappa$ is the
lower-corner element of $\mathcal{M}_{K_\alpha}$.

Let $v\in V_{h0}$. The coefficients $c_\alpha$, starting from $c_{(2,2,\ldots,2)}$, are determined successively in lexicographic order.  Suppose that the coefficients
$c_\beta$ for $\beta<\alpha$ have already been determined, and let 
$v_\alpha:=v-\sum_{\beta<\alpha}c_\beta\varphi_\beta$.
Among the patches containing $K_\kappa$, $\mathcal{M}_{K_\alpha}$
is the only one whose coefficient has not yet been determined.

By Definition~\ref{def:patch-function},
\[
\Sigma_{
E_{\alpha-\sum_{r=3}^{d}e_r}^{1,2}
}(\varphi_\alpha)>0.
\]
Hence there exists a unique $c_\alpha$ such that
\[
\Sigma_{
E_{\alpha-\sum_{r=3}^{d}e_r}^{1,2}
}
(v_\alpha-c_\alpha\varphi_\alpha)=0.
\]
All the remaining data required on the lower sides of $K_\kappa$
either vanish by the homogeneous boundary conditions or come from
elements already eliminated in the sweeping procedure. Step 1 in the proof of Lemma~\ref{lem:single-interior-zero} therefore gives
$(v_\alpha-c_\alpha\varphi_\alpha)|_{K_\kappa}=0$.

Proceeding through all $\kappa$ with
$1\leq\kappa_i\leq N_i-2$ gives
\[
\widetilde v
:=
v-\sum_{\alpha\in\IKhint}c_\alpha\varphi_\alpha,
\]
where $\widetilde v|_{K_\kappa}=0$ for all
$1\leq\kappa_i\leq N_i-2$. For each coordinate direction, the vanishing
on these elements provides the zero data required on one side of the
adjacent outer-layer elements, while the homogeneous boundary conditions
provide the corresponding zero data on the opposite side. Repeated
application of \textbf{Situation~1} in the proof of
Lemma~\ref{lem:single-interior-zero}, successively in each coordinate
direction, therefore propagates this vanishing to all remaining elements.
Hence $\widetilde v=0$ on $\Omega$, and consequently
\[
v=\sum_{\alpha\in\IKhint}c_\alpha\varphi_\alpha.
\]
Thus $\{\varphi_\alpha:\alpha\in\IKhint\}$ spans $V_{h0}$.

To prove linear independence, suppose that
$\sum_{\alpha\in\IKhint}c_\alpha\varphi_\alpha=0$.
Apply the same sweeping order. On $K_{(1,\ldots,1)}$,
$\varphi_{(2,\ldots,2)}$ is the only patch function whose coefficient
has not already been eliminated. Since its distinguished vertex-sum
value is nonzero, $c_{(2,\ldots,2)}=0$. Repeating the same argument,
at each $K_\kappa$ the only remaining patch function is
$\varphi_{\kappa+\mathbf1}$, and hence
$c_{\kappa+\mathbf1}=0$. Therefore $c_\alpha=0$ for every
$\alpha\in\IKhint$, proving linear independence.

Thus $\{\varphi_\alpha:\alpha\in\IKhint\}$ is a basis of $V_{h0}$,
and
\[
\dim V_{h0}=\#\Khint.
\]
If $\Khint=\varnothing$, then $V_{h0}=\{0\}$ by Remark~\ref{rem:no-interior-elements}, and the
same conclusion holds.
\end{proof}

\section{Interpolation and error estimates}\label{sec:basis-interpolation-error}

This section constructs the interpolation operator and uses it to derive the approximation and consistency estimates.

\subsection{Interpolation operator and optimal approximation}\label{sec:interpolation}
{
The derivation of the approximation estimate uses a quadratic-preserving
interpolation operator into \(V_{h0}\), generalizing the interpolation
operator of \cite{DaiZengZhangZhang2024} to arbitrary dimension. The construction is motivated by the considerations described in Remark~\ref{rem:interpolation-structure}.
\begin{remark}
\label{rem:interpolation-structure}
{
The construction combines three ingredients, each forced by a
constraint that cannot be relaxed: (i)~the natural dual of
$\varphi_\alpha$ is a point functional without uniform $L^2$
stability, so a stable realization by cell averages is mandatory;
(ii)~$P_2$ reproduction requires the full $3^d$-element patch
neighborhood of every element, whence the extended mesh;
(iii)~conformity with the boundary conditions of $V_{h0}$ forces the
restriction to interior centers, and the resulting boundary-layer
defect is controlled by the vanishing traces of $H_0^2(\Omega)$
functions.
}
\end{remark}

\subsubsection{Extended mesh and \(P_2\) reproduction}
\label{sec:interp-extended}

\begin{df}
\label{def:local-phi}
Let $\Kh$ be the original rectangular mesh. Extend each coordinate
grid by two cells on both sides by setting
\[
x_{i,0}:=2x_{i,1}-x_{i,2},
\qquad
x_{i,-1}:=x_{i,0}-(x_{i,2}-x_{i,1}),
\]
and
\[
x_{i,N_i+2}:=2x_{i,N_i+1}-x_{i,N_i},
\qquad
x_{i,N_i+3}:=x_{i,N_i+2}
 +(x_{i,N_i+1}-x_{i,N_i}).
\]
The two added cells at each end have the same length as the adjacent
boundary cell. Hence the extended mesh has the same shape regularity
as the original mesh.

Using the extended nodes, define $K_\alpha$ as before for
$-1\leq\alpha_i\leq N_i+2$, $1\leq i\leq d$, and set
\[
\widetilde{\mathcal I}_h^d
:=
\{\alpha=(\alpha_1,\ldots,\alpha_d):
-1\leq\alpha_i\leq N_i+2,\ 1\leq i\leq d\},
\]
and let the extended rectangular mesh be
\[
\widetilde{\Kh}
:=
\{K_\alpha:\alpha\in\widetilde{\mathcal I}_h^d\}.
\]
Its interior element index set is
\[
\widetilde{\mathcal I}_{h,\mathrm{int}}^d
:=
\{\alpha\in\widetilde{\mathcal I}_h^d:
0\leq\alpha_i\leq N_i+1,\ 1\leq i\leq d\}.
\]
For $\alpha\in\widetilde{\mathcal I}_{h,\mathrm{int}}^d$, define
$\mathcal M_{K_\alpha}$ as in Subsection~\ref{sec:basis-vh0} and set
\[
\widetilde\Omega_h
:=
\bigcup_{\alpha\in\widetilde{\mathcal I}_{h,\mathrm{int}}^d}
\mathcal M_{K_\alpha}
=
\prod_{i=1}^d[x_{i,-1},x_{i,N_i+3}].
\]
For each $\alpha\in\widetilde{\mathcal I}_{h,\mathrm{int}}^d$, let
$\varphi_\alpha$ be the patch function from
Definition~\ref{def:patch-function} on $\mathcal M_{K_\alpha}$,
extended by zero to $\widetilde\Omega_h$.
\end{df}

The following two local functionals are used to establish quadratic
polynomial reproduction. For
$\alpha\in\widetilde{\mathcal I}_{h,\mathrm{int}}^d$, let
$x_{K_\alpha}$ be the barycenter of $K_\alpha$ and, for convenience,
write $h_{\alpha,i}:=h_{i,\alpha_i}=x_{i,\alpha_i+1}-x_{i,\alpha_i}$.
For $v\in P_2(K_\alpha)$, define
\[
    r_{K_\alpha}(v)
    :=
    v(x_{K_\alpha})
    -\frac18\sum_{i=1}^{d}h_{\alpha,i}^{2}
    \frac{\partial^2 v}{\partial x_i^2},
\]
and
\[
    t_{K_\alpha}(v)
    :=
    \fint_{K_\alpha}v\,dx
    -\frac16\sum_{i=1}^{d}h_{\alpha,i}^{2}
    \frac{\partial^2 v}{\partial x_i^2}.
\]

\begin{lemma}
\label{lem:phi-basis-properties}
The patch functions
$\{\varphi_\alpha\}_{\alpha\in
\widetilde{\mathcal I}_{h,\mathrm{int}}^d}$ preserve  $P_2$ on $\Omega$:
for every $v\in P_2(\widetilde\Omega_h)$ and $x\in\Omega$,
\[
    I_hv(x)
    :=\sum_{\alpha\in\widetilde{\mathcal I}_{h,\mathrm{int}}^d}
    r_{K_\alpha}(v)\varphi_\alpha(x)=v(x),
    \qquad
    \sum_{\alpha\in\widetilde{\mathcal I}_{h,\mathrm{int}}^d}
    t_{K_\alpha}(v)\varphi_\alpha(x)=v(x).
\]
\end{lemma}

\begin{proof}
Fix an arbitrary element $K_\kappa\in\Kh$. On $K_\kappa$, only the
patch functions $\varphi_\alpha$ with
\[
\alpha_m\in\{\kappa_m-1,\kappa_m,\kappa_m+1\},
\qquad 1\leq m\leq d,
\]
can be nonzero. Their supports are contained in the $5^d$-cell patch
centered at $K_\kappa$, which is contained in $\widetilde\Omega_h$.
Since $K_\kappa$ is arbitrary, it is sufficient to prove
$
I_hv=v
$ on  $K_\kappa$. 

For $p\neq q$, direct calculation gives
\[
\begin{aligned}
r_{K_\alpha}(1)&=1,\\
r_{K_\alpha}(x_p)
&=\frac{x_{p,\alpha_p}+x_{p,\alpha_p+1}}2,\\
r_{K_\alpha}(x_px_q)
&=
\frac{x_{p,\alpha_p}+x_{p,\alpha_p+1}}2\cdot
\frac{x_{q,\alpha_q}+x_{q,\alpha_q+1}}2,\\
r_{K_\alpha}(x_p^2)
&=x_{p,\alpha_p}x_{p,\alpha_p+1}.
\end{aligned}
\]

The vertex sums and coordinate-derivative averages of the patch functions
on the faces of $K_\kappa$ are computed in
Appendix~\ref{subsec:details-phi-basis-properties}, together with the
elementary weight identities used in the calculation. These computations
yield
\begin{equation}
\label{eq:vertex-sum-reproduction}
\Sigma_{E_\gamma^{a,b}}(I_hv)
=
\Sigma_{E_\gamma^{a,b}}(v),
\qquad v\in P_2,
\end{equation}
and
\begin{equation}
\label{eq:derivative-average-reproduction}
\fint_{F_\beta^a}
\frac{\partial I_hv}{\partial x_a}\,ds
=
\fint_{F_\beta^a}
\frac{\partial v}{\partial x_a}\,ds,
\qquad v\in P_2.
\end{equation}

Thus $I_hv$ and $v$ have the same $(d-2)$-face vertex sums and
coordinate-derivative averages on $K_\kappa$. 
Theorem~\ref{lem:p2-data-compatibility} proves
$ 
I_hv=v 
$ on $K_\kappa$.
Since $K_\kappa$ is arbitrary, the first identity holds in 
$\Omega$.

The exact quadratic expansion about $x_{K_\alpha}$ gives
\begin{equation}
\label{eq:quadratic-cell-average}
\fint_{K_\alpha}v\,dx
=
v(x_{K_\alpha})
+\frac1{24}\sum_{i=1}^d h_{\alpha,i}^2
\frac{\partial^2v}{\partial x_i^2}.
\end{equation}
Hence
$t_{K_\alpha}(v)=r_{K_\alpha}(v)$ for every
$v\in P_2(\widetilde\Omega_h)$, and the second identity follows
from the first one.
\end{proof}

A uniform local bound for the patch functions will be used in the
stability estimate for the interpolation operator.

\begin{lemma}
\label{lem:phi-local-bound}
For every $\alpha\in\widetilde{\mathcal I}_{h,\mathrm{int}}^d$ and
$k\geq0$,
\[
    |\varphi_\alpha|_{k,T}\leq C h_T^{d/2-k},
    \qquad
    \text{for all }T\in\mathcal{M}_{K_\alpha},
\]
where $C>0$ is independent of $h$.
\end{lemma}

\begin{proof}  
Theorem~\ref{thm:rm-unisolvence} shows that, on $P_2(T)$, the sum of the absolute values of all $(d-2)$-face vertex-sum functionals and all coordinate derivative-average functionals defines a norm.
The result follows from the standard scaling argument, finite-dimensional
norm equivalence on the reference rectangle, and the uniform boundedness
of the degrees of freedom given by
\eqref{eq:general-patch-vertex-sum} and
\eqref{eq:general-patch-derivative-average}. The details are omitted.
\end{proof}
\subsubsection{Stable local averaging and the auxiliary interpolant}
\label{sec:interp-auxiliary}

Although \(r_{K_\alpha}\) and \(t_{K_\alpha}\) reproduce \(P_2\), they
involve point evaluations or second derivatives and are therefore not
suitable for stability estimates. To obtain a stable interpolation operator,
\(t_{K_\alpha}\) is replaced by an \(L^1\)-stable local averaging functional
that agrees with it on \(P_2\). These functionals are then used to define an
auxiliary interpolant on the extended mesh.

\begin{df}
\label{def:auxiliary-interpolator}
Recall $e_i$ be the $i$-th coordinate vector. For every
$\alpha\in\widetilde{\mathcal I}_{h,\mathrm{int}}^d$, define
\begin{equation}
\label{eq:lambda-definition}
\begin{aligned}
\lambda_{K_\alpha}(v)
:= w_{\alpha,0}\fint_{K_\alpha}v\,dx 
 +\sum_{i=1}^d\left(
w_{\alpha,i}^{-}\fint_{K_{\alpha-e_i}}v\,dx
+w_{\alpha,i}^{+}\fint_{K_{\alpha+e_i}}v\,dx
\right),
\end{aligned}
\end{equation}
where
\begin{align*}
w_{\alpha,i}^{-}
&:=
-\frac{h_{\alpha,i}^2}
{(h_{\alpha-e_i,i}+h_{\alpha,i})
 (h_{\alpha-e_i,i}+h_{\alpha,i}+h_{\alpha+e_i,i})},\\
w_{\alpha,i}^{+}
&:=
-\frac{h_{\alpha,i}^2}
{(h_{\alpha,i}+h_{\alpha+e_i,i})
 (h_{\alpha-e_i,i}+h_{\alpha,i}+h_{\alpha+e_i,i})},
\qquad 1\leq i\leq d,
\end{align*}
and
\[
w_{\alpha,0}
:=1-\sum_{i=1}^d
\bigl(w_{\alpha,i}^{-}+w_{\alpha,i}^{+}\bigr).
\]
The weights $w_{\alpha,i}^{-}$ and $w_{\alpha,i}^{+}$ in
\eqref{eq:lambda-definition} are the unique choice for which the
first-derivative terms in the neighboring cell-average differences cancel
and the second-derivative terms match $-\frac16h_{\alpha,i}^{2}$;
see the proof of Lemma~\ref{lem:lambda-t-identity} in
Appendix~\ref{subsec:proof-lambda-t}.
Thus $\lambda_{K_\alpha}(v)$ depends only on the cell averages over
$K_\alpha$ and the $2d$ neighboring elements $K_{\alpha\pm e_i}$.

Define the auxiliary interpolation operator
$\widetilde\Pi_h:L^1(\widetilde\Omega_h)\longrightarrow
\operatorname{span}\{\varphi_\alpha:\alpha\in
\widetilde{\mathcal I}_{h,\mathrm{int}}^d\}$ by
\begin{equation}
\label{eq:auxiliary-interpolator}
\widetilde\Pi_hv
:=\sum_{\alpha\in\widetilde{\mathcal I}_{h,\mathrm{int}}^d}
\lambda_{K_\alpha}(v)\varphi_\alpha.
\end{equation}
For $T=K_\kappa\in\Kh$, define
\begin{equation}
\label{eq:delta-T}
\Delta_T
:=\bigcup_{\delta\in\{-2,-1,0,1,2\}^d}
K_{\kappa+\sum_{i=1}^d\delta_i e_i},
\end{equation}
the $5^d$-element rectangular patch centered at $T$; by the stencil
in \eqref{eq:lambda-definition}, $\widetilde\Pi_hv|_T$ depends only
on $v|_{\Delta_T}$.
\end{df}

\begin{lemma}
\label{lem:lambda-t-identity}
For $\alpha\in\widetilde{\mathcal I}_{h,\mathrm{int}}^d$, set
\[
\mathcal S_{K_\alpha}
:=K_\alpha\cup\bigcup_{i=1}^d
\bigl(K_{\alpha-e_i}\cup K_{\alpha+e_i}\bigr).
\]
If $p\in P_2(\mathcal S_{K_\alpha})$, then
\[
\lambda_{K_\alpha}(p)=t_{K_\alpha}(p).
\]
\end{lemma}

The proof is given in Appendix~\ref{subsec:proof-lambda-t}.
The previous identity yields polynomial  preservation, stability, and
local approximation for $\widetilde\Pi_h$.

\begin{lemma}
\label{lem:auxiliary-interpolator}
For every $T\in\Kh$, the following properties hold.
\begin{enumerate}[label=(\alph*)]
\item If $v\in L^1(\widetilde\Omega_h)$ and
$v|_{\Delta_T}\in P_2(\Delta_T)$, then
\[
(\widetilde\Pi_hv)|_T=v|_T.
\]
\item For $0\leq k\leq2$,
\[
|\widetilde\Pi_hv|_{k,T}
\leq C h_T^{-k}\|v\|_{0,\Delta_T}.
\]
\item If $v\in H^s(\Delta_T)$, then
\[
|v-\widetilde\Pi_hv|_{k,T}
\leq C h_T^{s-k}|v|_{s,\Delta_T},
\qquad 0\leq k\leq s\leq3.
\]
\end{enumerate}
The constant $C$ is independent of $h$ and $T$.
\end{lemma}

\begin{proof}
Let $p\in P_2(\Delta_T)$. For every patch function contributing on
$T$, the stencil $\mathcal S_{K_\alpha}$ is contained in $\Delta_T$.
Thus Lemma~\ref{lem:lambda-t-identity} gives
$\lambda_{K_\alpha}(p)=t_{K_\alpha}(p)$, and
Lemma~\ref{lem:phi-basis-properties} yields
$(\widetilde\Pi_hp)|_T=p|_T$. This proves (a).

For (b), shape regularity and \eqref{eq:lambda-definition} imply
\[
|\lambda_{K_\alpha}(v)|
\leq C h_T^{-d/2}\|v\|_{0,\mathcal S_{K_\alpha}}
\]
for every $\alpha$ with $T\subset\mathcal M_{K_\alpha}$.
Since only $3^d$ such indices occur, Lemma~\ref{lem:phi-local-bound}
and the uniformly bounded overlap of the stencils give
\[
|\widetilde\Pi_hv|_{k,T}
\leq C h_T^{-k}\|v\|_{0,\Delta_T},
\qquad 0\leq k\leq2.
\]

Finally, for any $p\in P_2(\Delta_T)$,  (a) and (b) imply
\[
|v-\widetilde\Pi_hv|_{k,T}
\leq |v-p|_{k,T}
+C h_T^{-k}\|v-p\|_{0,\Delta_T},
\]
and the Bramble--Hilbert lemma \cite{brennerscott} on the uniformly
shape-regular patch $\Delta_T$ gives
\[
|v-\widetilde\Pi_hv|_{k,T}
\leq C h_T^{s-k}|v|_{s,\Delta_T},
\qquad 0\leq k\leq s\leq3.
\]
For $k>2$, the claim is trivial since
$\widetilde\Pi_hv|_T\in P_2(T)$. This proves (c).
\end{proof}
\subsubsection{Homogeneous interpolant}
\label{sec:interp-homogeneous}

The auxiliary operator contains contributions from patch functions
centered outside the original interior index set. Omitting these
contributions yields an operator mapping into $V_{h0}$.

\begin{df}
\label{def:pi-h0}
Define the homogeneous interpolation operator by
\begin{equation}
\label{eq:pi-h0-definition}
\Pi_{h0}:L^1(\Omega)\longrightarrow V_{h0},
\qquad
\Pi_{h0}v
:=\sum_{\alpha\in\IKhint}
\lambda_{K_\alpha}(v)\varphi_\alpha.
\end{equation}
\end{df}
{
Each coefficient $\lambda_{K_\alpha}(v)$ uses only the $2d+1$ cells in
\eqref{eq:lambda-definition}; since $\Pi_{h0}$ involves only cells of
the original mesh, $\widetilde\Pi_h$ and $\Pi_{h0}$ can differ only
near $\partial\Omega$.
}

\subsubsection{Approximation estimate}
\label{sec:interp-approx}

The local estimate for $\widetilde\Pi_h$ now transfers to
$\Pi_{h0}$. Only the boundary layer requires an additional argument. 

\begin{theorem}
\label{thm:pi-h0-error}
For $v\in H_0^2(\Omega)\cap H^3(\Omega)$,
\[
|v-\Pi_{h0}v|_{k,h}
\leq C h^{3-k}|v|_{3,\Omega},
\qquad
0\leq k\leq3,
\]
where $C$ is independent of $h$ and $v$.
\end{theorem}

The boundary layer is controlled by the following lemma, which isolates
the trace argument. Its proof is deferred to
Appendix~\ref{subsec:boundary-trace-proof}. 
Finally, since $\Omega$ is Lipschitz,
\cite[Theorem~1.4.5]{brennerscott} provides an extension operator
$\mathscr X$ such that, for $0\leq s\leq3$,
\begin{equation}
\label{eq:sobolev-extension}
(\mathscr Xv)|_\Omega=v,
\qquad
\|\mathscr Xv\|_{s,\mathbb R^d}\leq C\|v\|_{s,\Omega}.
\end{equation}

\begin{lemma}
\label{lem:boundary-trace}
{
Let $v\in H_0^2(\Omega)\cap H^3(\Omega)$, let $\mathscr X$ be the
extension operator in \eqref{eq:sobolev-extension}, and let $T\in\Kh$.
For $\alpha\in\mathcal A_T:=\{\alpha\in
\widetilde{\mathcal I}_{h,\mathrm{int}}^d\setminus\IKhint:
T\subset\mathcal M_{K_\alpha}\}$, let $a_\alpha$ be a vertex of
$K_\alpha$ on $\partial\Omega$; such a vertex exists because
$\alpha\notin\IKhint$. Then, for every $p\in P_2(\Delta_T)$,
}
\begin{equation}
\label{eq:boundary-taylor}
\lambda_{K_\alpha}(p)
=
p(a_\alpha)
+\sum_{i=1}^d
\delta_{\alpha,i}\frac{h_{\alpha,i}}2
\frac{\partial p}{\partial x_i}(a_\alpha)
+\sum_{1\leq i<j\leq d}
\delta_{\alpha,i}\delta_{\alpha,j}
\frac{h_{\alpha,i}h_{\alpha,j}}4
\frac{\partial^2p}{\partial x_i\partial x_j}(a_\alpha),
\end{equation}
where $\delta_{\alpha,i}\in\{-1,1\}$. Consequently, if $p$ is chosen as
the local quadratic Bramble--Hilbert approximation of $\mathscr Xv$ on
$\Delta_T$ \cite{brennerscott}, then
\begin{equation}
\label{eq:boundary-coefficient-s3}
|\lambda_{K_\alpha}(p)|
\leq C h_T^{3-d/2}|\mathscr Xv|_{3,\Delta_T}.
\end{equation}
The proof is given in Appendix~\ref{subsec:boundary-trace-proof}.
\end{lemma}

\begin{proof}[Proof of Theorem~\ref{thm:pi-h0-error}]
Let $\mathscr X$ be the extension operator in
\eqref{eq:sobolev-extension}. Lemma~\ref{lem:auxiliary-interpolator}(c),
finite overlap, and \eqref{eq:sobolev-extension} give
\begin{equation}
\label{eq:auxiliary-global-error}
|\mathscr Xv-\widetilde\Pi_h\mathscr Xv|_{k,h}
\leq C h^{3-k}|v|_{3,\Omega}.
\end{equation}
Hence it remains to estimate
$\widetilde\Pi_h\mathscr Xv-\Pi_{h0}v$, which is supported only in
the boundary layer.

For $T\in\Kh$, let $\mathcal A_T$ be as in
Lemma~\ref{lem:boundary-trace}. 
If $\mathcal A_T=\varnothing$, the two interpolants coincide on $T$;
in particular, by Lemma~\ref{lem:auxiliary-interpolator}(a),
$\Pi_{h0}v=v$ on $T$ whenever $v|_{\Delta_T}\in P_2(\Delta_T)$.
Otherwise,
\begin{equation}
\label{eq:boundary-correction}
(\widetilde\Pi_h\mathscr Xv-\Pi_{h0}v)|_T
=
\sum_{\alpha\in\mathcal A_T}
\lambda_{K_\alpha}(\mathscr Xv)\varphi_\alpha|_T.
\end{equation}
For any $p\in P_2(\Delta_T)$, Lemma~\ref{lem:phi-local-bound} and the
local stability argument above give
\begin{equation}
\label{eq:boundary-splitting}
\left|
\widetilde\Pi_h\mathscr Xv-\Pi_{h0}v
\right|_{k,T}
\leq
C h_T^{-k}\|\mathscr Xv-p\|_{0,\Delta_T}
+
C h_T^{d/2-k}
\left(\sum_{\alpha\in\mathcal A_T}
|\lambda_{K_\alpha}(p)|^2\right)^{1/2}.
\end{equation}

Choosing $p$ as in Lemma~\ref{lem:boundary-trace}, substituting
\eqref{eq:boundary-coefficient-s3} and the same Bramble--Hilbert
estimate \cite{brennerscott} into \eqref{eq:boundary-splitting} yields
\[
|\widetilde\Pi_h\mathscr Xv-\Pi_{h0}v|_{k,T}
\leq C h_T^{3-k}|\mathscr Xv|_{3,\Delta_T}.
\]
Summing over the boundary elements and using finite overlap gives
\begin{equation}
\label{eq:boundary-global-s3}
|\widetilde\Pi_h\mathscr Xv-\Pi_{h0}v|_{k,h}
\leq C h^{3-k}|v|_{3,\Omega}.
\end{equation}
Combining this estimate with \eqref{eq:auxiliary-global-error} completes
the proof for $0\le k\le2$. For $k=3$, since
$\Pi_{h0}v|_T\in P_2(T)$, one has
$|v-\Pi_{h0}v|_{3,T}=|v|_{3,T}$, and the estimate is trivial.
\end{proof}

\subsection{Consistency error estimation}\label{sec:consistency-error}

The interpolation estimate is complemented by a consistency estimate.
The key observation is that the vertex-sum continuity also controls the
averages of the tangential derivatives across element faces.

\begin{lemma}
\label{lem:gradient-average-continuity}
Let $v_h\in V_{h0}$.  For every interior face
$F_\beta^i\in\Fhddint$, whose two adjacent elements are
$K_\beta$ and $K_{\beta-e_i}$, one has
\begin{equation}
\label{eq:interior-gradient-average-jump}
\int_{F_\beta^i}
\left(
\nabla(v_h|_{K_\beta})-
\nabla(v_h|_{K_{\beta-e_i}})
\right)ds=0.
\end{equation}
If $F_\beta^i\in\Fhddb$ and $K_\alpha$ is its unique adjacent
element, then
\begin{equation}
\label{eq:boundary-gradient-average}
\int_{F_\beta^i}\nabla(v_h|_{K_\alpha})\,ds=0.
\end{equation}
\end{lemma}

\begin{proof}
Recall that $x_{F_\beta^i}$ is the barycenter of $F_\beta^i$. For a
tangential direction $j\ne i$ and every $p\in P_2(K_\alpha)$, where
$F_\beta^i$ is a face of $K_\alpha$, the following identity holds:
\begin{equation}
\label{eq:tangential-average-by-face-sums}
\fint_{F_\beta^i}\frac{\partial p}{\partial x_j}\,ds
=\frac{\partial p}{\partial x_j}(x_{F_\beta^i})
=\frac{
\vsum{E_{\beta+e_j}^{i,j}}(p)-\vsum{E_\beta^{i,j}}(p)
}{2^{d-2}h_{j,\beta_j}}.
\end{equation}
The first equality follows because $\partial p/\partial x_j$ is affine
on $F_\beta^i$. For each $z\in\mathcal V(E_\beta^{i,j})$, the exact
centered-difference formula in the $x_j$-direction gives
\[
p\left(z+h_{j,\beta_j}e_j\right)-p(z)
=
h_{j,\beta_j}
\frac{\partial p}{\partial x_j}
\left(z+\frac{h_{j,\beta_j}}2e_j\right).
\]
Summing over $z\in\mathcal V(E_\beta^{i,j})$ and applying
Lemma~\ref{lem:p1-face-sum-vertex} to the resulting $(d-2)$-dimensional
face lead to \eqref{eq:tangential-average-by-face-sums}.

Equation~\eqref{eq:tangential-average-by-face-sums} and the
continuity  of the vertex-sums across $(d-2)$-faces imply that the integrals of all tangential
derivatives are continuous across interior faces and vanish on
boundary faces. By the definition of $V_{h0}$, the same property
holds for the normal derivative. Hence
\eqref{eq:interior-gradient-average-jump} and
\eqref{eq:boundary-gradient-average} follow.
\end{proof}
\begin{remark}
The proof of Lemma~\ref{lem:gradient-average-continuity} shows that the continuity conditions defining $V_h$ and $V_{h0}$ in \eqref{ncp2-space} and \eqref{ncp2-space-homogeneous}, respectively, can be equivalently expressed in terms of the $(d-1)$-face vertex sums and the face-averaged gradients.
\end{remark}\begin{theorem}[Consistency error estimate]
\label{thm:consistency-error}
Let $\Kh$ be the shape-regular rectangular mesh defined above.
There exists a constant $C$, independent of $h$, $v$, and $v_h$,
such that, for every
$v\in H_0^2(\Omega)\cap H^3(\Omega)$ satisfying
$\Delta^2v\in L^2(\Omega)$,
\begin{equation}
\label{eq:consistency-error}
\Big|
\sum_{K\in\Kh}(\nabla^2v,\nabla^2v_h)_K
-(\Delta^2v,v_h)
\Big|
\leq Ch\left(
|v|_{3,\Omega}+h\|\Delta^2v\|_{0,\Omega}
\right)|v_h|_{2,h},
\quad \forall v_h\in V_{h0}.
\end{equation}
\end{theorem}

\begin{proof}
Lemma~\ref{lem:gradient-average-continuity} provides the continuity of
the face-averaged derivatives required in the standard consistency argument. Elementwise
integration by parts, together with the trace estimate as in
\cite[Lemma~7]{WangXu2006}, yields \eqref{eq:consistency-error}.
\end{proof}

\appendix
\section{Auxiliary proofs}

\subsection{Proof details for Lemma~\ref{lem:phi-basis-properties}}
\label{subsec:details-phi-basis-properties}

\begin{proof}
The proof begins by collecting the computational ingredients needed for
Lemma~\ref{lem:phi-basis-properties}.
The weights satisfy, for any coordinate direction $m$ and any $1\leq\rho\leq N_m+1$,
\begin{equation}
\label{eq:weight-fixed-identities}
\begin{aligned}
&\sum_{\alpha_m\in\{\rho-1,\rho\}}
\omega_{\alpha;m,\eta_{\alpha,m}(\rho)}=1,\\
&\sum_{\alpha_m\in\{\rho-1,\rho\}}
\omega_{\alpha;m,\eta_{\alpha,m}(\rho)}
\frac{x_{m,\alpha_m}+x_{m,\alpha_m+1}}2
=x_{m,\rho},\\
&\sum_{\alpha_m\in\{\rho-1,\rho\}}
\omega_{\alpha;m,\eta_{\alpha,m}(\rho)}
x_{m,\alpha_m}x_{m,\alpha_m+1}
=x_{m,\rho}^2,
\end{aligned}
\end{equation}
and any $1\leq\rho\leq N_m$,
\begin{equation}
\label{eq:weight-cell-identities}
\begin{aligned}
&\sum_{\alpha_m=\rho-1}^{\rho+1}
\omega_{\alpha;m,\rho}=2,\\
&\sum_{\alpha_m=\rho-1}^{\rho+1}
\omega_{\alpha;m,\rho}
\frac{x_{m,\alpha_m}+x_{m,\alpha_m+1}}2
=x_{m,\rho}+x_{m,\rho+1},\\
&\sum_{\alpha_m=\rho-1}^{\rho+1}
\omega_{\alpha;m,\rho}
x_{m,\alpha_m}x_{m,\alpha_m+1}
=x_{m,\rho}^2+x_{m,\rho+1}^2.
\end{aligned}
\end{equation}

Let $E_\gamma^{a,b}$ be a $(d-2)$-face of $K_\kappa$, with $a<b$.
By \eqref{eq:general-patch-vertex-sum},
\begin{equation}
\label{eq:vertex-sum-interpolant}
\begin{aligned}
\Sigma_{E_\gamma^{a,b}}(I_hv)
={}&
\sum_{\substack{
\alpha_a\in\{\gamma_a-1,\gamma_a\},\
\alpha_b\in\{\gamma_b-1,\gamma_b\}\\
\alpha_m\in\{\gamma_m-1,\gamma_m,\gamma_m+1\},
\ m\notin\{a,b\}}}
r_{K_\alpha}(v)\,
\omega_{\alpha;a,\eta_{\alpha,a}(\gamma_a)}
\omega_{\alpha;b,\eta_{\alpha,b}(\gamma_b)}
\prod_{m\notin\{a,b\}}\omega_{\alpha;m,\gamma_m}.
\end{aligned}
\end{equation}
Similarly, for a $(d-1)$-face $F_\beta^a$ of $K_\kappa$,
\eqref{eq:general-patch-derivative-average} gives
\begin{equation}
\label{eq:derivative-average-interpolant}
\begin{aligned}
\fint_{F_\beta^a}\frac{\partial I_hv}{\partial x_a}\,ds
={}&
\frac{1}{
2^{d-2}(h_{a,\beta_a-1}+h_{a,\beta_a})
}
\sum_{\substack{
\alpha_a=\beta_a-1\\
\alpha_m\in\{\beta_m-1,\beta_m,\beta_m+1\},\
m\neq a}}
\bigl(r_{K_{\alpha+e_a}}(v)-r_{K_\alpha}(v)\bigr)
\prod_{m\neq a}\omega_{\alpha;m,\beta_m}.
\end{aligned}
\end{equation}

To verify \eqref{eq:vertex-sum-reproduction}, evaluate
\eqref{eq:vertex-sum-interpolant} on the monomial basis of $P_2$.

\emph{The case $v=1$.}
Using the first identities in \eqref{eq:weight-fixed-identities} and
\eqref{eq:weight-cell-identities} gives
\[
\Sigma_{E_\gamma^{a,b}}(I_h1)
=1\cdot1\cdot\prod_{m\notin\{a,b\}}2
=2^{d-2}
=\Sigma_{E_\gamma^{a,b}}(1).
\]

\emph{The case $v=x_p$.}
If $p\in\{a,b\}$, then \eqref{eq:weight-fixed-identities} and
\eqref{eq:weight-cell-identities} give
\[
\Sigma_{E_\gamma^{a,b}}(I_hx_p)
=2^{d-2}x_{p,\gamma_p}
=\Sigma_{E_\gamma^{a,b}}(x_p).
\]
If $p\notin\{a,b\}$, then
\[
\begin{aligned}
\Sigma_{E_\gamma^{a,b}}(I_hx_p)
&=2^{d-3}
\bigl(x_{p,\gamma_p}+x_{p,\gamma_p+1}\bigr)=\Sigma_{E_\gamma^{a,b}}(x_p).
\end{aligned}
\]
Indeed, one half of the vertices have $p$-th coordinate
$x_{p,\gamma_p}$, and the other half have $p$-th coordinate
$x_{p,\gamma_p+1}$.

\emph{The case $v=x_px_q$, $p\neq q$.}
If $p,q\in\{a,b\}$, then
\[
\Sigma_{E_\gamma^{a,b}}(I_h(x_px_q))
=2^{d-2}x_{p,\gamma_p}x_{q,\gamma_q}
=\Sigma_{E_\gamma^{a,b}}(x_px_q).
\]
If $p\in\{a,b\}$ and $q\notin\{a,b\}$, then
\[
\begin{aligned}
\Sigma_{E_\gamma^{a,b}}(I_h(x_px_q))=2^{d-3}x_{p,\gamma_p}
\bigl(x_{q,\gamma_q}+x_{q,\gamma_q+1}\bigr)=\Sigma_{E_\gamma^{a,b}}(x_px_q).
\end{aligned}
\]
The case $q\in\{a,b\}$ and $p\notin\{a,b\}$ is analogous. Finally,
if $p,q\notin\{a,b\}$, then
\[
\Sigma_{E_\gamma^{a,b}}(I_h(x_px_q))
=2^{d-4}
\bigl(x_{p,\gamma_p}+x_{p,\gamma_p+1}\bigr)
\bigl(x_{q,\gamma_q}+x_{q,\gamma_q+1}\bigr)
=\Sigma_{E_\gamma^{a,b}}(x_px_q).
\]
The product of the separate coordinate sums is precisely the sum of
$x_px_q$ over all vertices of $E_\gamma^{a,b}$.

\emph{The case $v=x_p^2$.}
If $p\in\{a,b\}$, then the last identity in \eqref{eq:weight-fixed-identities} gives
\[
\Sigma_{E_\gamma^{a,b}}(I_h(x_p^2))
=2^{d-2}x_{p,\gamma_p}^2
=\Sigma_{E_\gamma^{a,b}}(x_p^2).
\]
If $p\notin\{a,b\}$, then the last identity in \eqref{eq:weight-cell-identities} gives
\[
\begin{aligned}
\Sigma_{E_\gamma^{a,b}}(I_h(x_p^2))
&=2^{d-3}
\bigl(x_{p,\gamma_p}^2+x_{p,\gamma_p+1}^2\bigr)=\Sigma_{E_\gamma^{a,b}}(x_p^2).
\end{aligned}
\]
This proves \eqref{eq:vertex-sum-reproduction} by linearity.

Next verify \eqref{eq:derivative-average-reproduction} by evaluating
\eqref{eq:derivative-average-interpolant} on the same monomial basis.

\emph{The case $v=1$.}
Since
\[
r_{K_{\alpha+e_a}}(1)-r_{K_\alpha}(1)=0,
\]
equation \eqref{eq:derivative-average-interpolant} gives
\[
\fint_{F_\beta^a}\frac{\partial I_h1}{\partial x_a}\,ds
=0
=\fint_{F_\beta^a}\frac{\partial 1}{\partial x_a}\,ds.
\]

\emph{The case $v=x_p$.}
If $p\neq a$, the coefficient $r_{K_\alpha}(x_p)$ is unchanged when
$\alpha$ is replaced by $\alpha+e_a$, and hence both derivative
averages are zero. If $p=a$, then
\[
\begin{aligned}
r_{K_{\alpha+e_a}}(x_a)-r_{K_\alpha}(x_a)
&=
\frac{x_{a,\beta_a}+x_{a,\beta_a+1}}2
-\frac{x_{a,\beta_a-1}+x_{a,\beta_a}}2=\frac{h_{a,\beta_a-1}+h_{a,\beta_a}}2.
\end{aligned}
\]
Furthermore, the first identity in
\eqref{eq:weight-cell-identities} gives
\[
\sum_{\substack{
\alpha_m\in\{\beta_m-1,\beta_m,\beta_m+1\}\\m\neq a}}
\prod_{m\neq a}\omega_{\alpha;m,\beta_m}
=2^{d-1}.
\]
Substitution into \eqref{eq:derivative-average-interpolant} yields
\[
\fint_{F_\beta^a}
\frac{\partial I_hx_a}{\partial x_a}\,ds
=1
=\fint_{F_\beta^a}
\frac{\partial x_a}{\partial x_a}\,ds.
\]

\emph{The case $v=x_px_q$, $p\neq q$.}
If $a\notin\{p,q\}$, then the difference of the two coefficients in
\eqref{eq:derivative-average-interpolant} vanishes, and both
derivative averages vanish. If $a=p$,
then
\[
\begin{aligned}
&r_{K_{\alpha+e_a}}(x_ax_q)-r_{K_\alpha}(x_ax_q)=
\frac{h_{a,\beta_a-1}+h_{a,\beta_a}}2
\frac{x_{q,\alpha_q}+x_{q,\alpha_q+1}}2.
\end{aligned}
\]
Using the second identity in \eqref{eq:weight-cell-identities} in the
$q$-direction and the first identity in the remaining $d-2$
directions, \eqref{eq:derivative-average-interpolant} shows
\[
\begin{aligned}
\fint_{F_\beta^a}
\frac{\partial I_h(x_ax_q)}{\partial x_a}\,ds=
\frac{x_{q,\beta_q}+x_{q,\beta_q+1}}2 =
\fint_{F_\beta^a}x_q\,ds
=
\fint_{F_\beta^a}
\frac{\partial(x_ax_q)}{\partial x_a}\,ds.
\end{aligned}
\]
{
The case $a=q$ is symmetric.
}

\emph{The case $v=x_p^2$.}
If $p\neq a$, the difference in
\eqref{eq:derivative-average-interpolant} again vanishes. If $p=a$,
then
\[
\begin{aligned}
r_{K_{\alpha+e_a}}(x_a^2)-r_{K_\alpha}(x_a^2)& =
x_{a,\beta_a}x_{a,\beta_a+1}
-x_{a,\beta_a-1}x_{a,\beta_a}\\
&=
x_{a,\beta_a}
\bigl(h_{a,\beta_a-1}+h_{a,\beta_a}\bigr).
\end{aligned}
\]
Together with the factor $2^{d-1}$ from the tangential directions,
 \eqref{eq:derivative-average-interpolant} gives
\[
\begin{aligned}
\fint_{F_\beta^a}
\frac{\partial I_h(x_a^2)}{\partial x_a}\,ds=2x_{a,\beta_a}=
\fint_{F_\beta^a}
\frac{\partial x_a^2}{\partial x_a}\,ds.
\end{aligned}
\]
This proves \eqref{eq:derivative-average-reproduction} by linearity.
\end{proof}

\subsection{Proof of Lemma~\ref{lem:lambda-t-identity}}
\label{subsec:proof-lambda-t}

\begin{proof}
Using \eqref{eq:quadratic-cell-average} and the exact Taylor expansion
at $x_{K_\alpha}$, for each coordinate direction $i$,
\begin{align*}
\fint_{K_{\alpha-e_i}}p\,dx-\fint_{K_\alpha}p\,dx
&=-\frac{h_{\alpha-e_i,i}+h_{\alpha,i}}2
\frac{\partial p}{\partial x_i}(x_{K_\alpha})
+\frac{(h_{\alpha-e_i,i}+h_{\alpha,i})(2h_{\alpha-e_i,i}+h_{\alpha,i})}{12}
\frac{\partial^2p}{\partial x_i^2},\\
\fint_{K_{\alpha+e_i}}p\,dx-\fint_{K_\alpha}p\,dx
&=\frac{h_{\alpha,i}+h_{\alpha+e_i,i}}2
\frac{\partial p}{\partial x_i}(x_{K_\alpha})
+\frac{(h_{\alpha,i}+h_{\alpha+e_i,i})(h_{\alpha,i}+2h_{\alpha+e_i,i})}{12}
\frac{\partial^2p}{\partial x_i^2}.
\end{align*}
Substituting $w_{\alpha,i}^{-}$ and $w_{\alpha,i}^{+}$,
the coefficient of the first derivative is
\[
-\frac{h_{\alpha-e_i,i}+h_{\alpha,i}}2w_{\alpha,i}^{-}
+\frac{h_{\alpha,i}+h_{\alpha+e_i,i}}2w_{\alpha,i}^{+}
=0,
\]
and the coefficient of the second derivative is
\begin{align*}
&\frac{(h_{\alpha-e_i,i}+h_{\alpha,i})
(2h_{\alpha-e_i,i}+h_{\alpha,i})}{12}w_{\alpha,i}^{-}+
\frac{(h_{\alpha,i}+h_{\alpha+e_i,i})
(h_{\alpha,i}+2h_{\alpha+e_i,i})}{12}w_{\alpha,i}^{+}
=-\frac{h_{\alpha,i}^2}{6}.
\end{align*}
The definition of $w_{\alpha,0}$ and summing the previous
coefficient identities over $i$ gives
\[
\lambda_{K_\alpha}(p)
=\fint_{K_\alpha}p\,dx
-\frac16\sum_{i=1}^d h_{\alpha,i}^2
\frac{\partial^2p}{\partial x_i^2}
=t_{K_\alpha}(p),
\]
as required.
\end{proof}
\subsection{Proof of Lemma~\ref{lem:boundary-trace}}
\label{subsec:boundary-trace-proof}

\begin{proof}
{
By Lemma~\ref{lem:lambda-t-identity} and
Lemma~\ref{lem:phi-basis-properties},
$\lambda_{K_\alpha}(p)=r_{K_\alpha}(p)$ for $p\in P_2(\Delta_T)$. Since
$x_{K_\alpha}-a_\alpha=\sum_{i=1}^d\delta_{\alpha,i}
\frac{h_{\alpha,i}}2e_i$ with $\delta_{\alpha,i}\in\{-1,1\}$, the
Taylor expansion of $p$ at $a_\alpha$ gives
\eqref{eq:boundary-taylor}; the pure second-derivative terms cancel in
$r_{K_\alpha}$.

The argument uses two standard estimates. By
\cite[Lemma~2]{clement1975approximation},
if $F$ is a $(d-1)$-dimensional rectangular face and
$p\in P_\ell(F)$, $\ell\geq0$, then
}
\begin{equation}
\label{eq:face-inverse}
\|p\|_{0,\infty,F}^2
\leq C|F|^{-1}\|p\|_{0,F}^2.
\end{equation}
An argument analogous to that in
\cite[pp.~24--25]{fichera1965linear} shows that, if $F$ is a face of
a rectangular element $K$ and $z\in H^1(K)$, then
\begin{equation}
\label{eq:element-trace}
\|z\|_{0,F}^2
\leq C\left(
h_K^{-1}\|z\|_{0,K}^2+h_K|z|_{1,K}^2
\right).
\end{equation}
Since $v\in H_0^2(\Omega)$, the traces of $v$, its first derivatives,
and the mixed second derivatives appearing in
\eqref{eq:boundary-taylor} vanish on the relevant boundary faces.
Thus \eqref{eq:face-inverse}--\eqref{eq:element-trace} imply, for
each term $D^jp$, $0\leq j\leq2$, appearing in
\eqref{eq:boundary-taylor},
\begin{equation}
\label{eq:boundary-trace-chain}
h_T^{2j}\|D^jp\|_{0,\infty,F}^2
\leq C\left(
h_T^{2j-d}\|\mathscr Xv-p\|_{j,\Delta_T}^2
+h_T^{2j-d+2}|\mathscr Xv-p|_{j+1,\Delta_T}^2
\right).
\end{equation}
Choosing $p$ as the local quadratic Bramble--Hilbert approximation
\cite{brennerscott} of $\mathscr Xv$ on $\Delta_T$ and substituting
\eqref{eq:boundary-trace-chain} into \eqref{eq:boundary-taylor} gives
\eqref{eq:boundary-coefficient-s3}.
\end{proof}
\bibliographystyle{plain}
\bibliography{bibfile_abbreviated}

\end{document}